\documentclass[11pt]{article}

\usepackage[a4paper,margin=1in]{geometry}
\usepackage{amsmath,amssymb,amsthm,mathtools}
\usepackage{enumitem}
\usepackage{booktabs}
\usepackage{microtype}
\usepackage{mathrsfs}
\usepackage{tikz}
\usepackage{algorithm}
\usepackage{algorithmic}
\usepackage[numbers,sort&compress]{natbib}
\usepackage[hidelinks]{hyperref}
\providecommand{\doi}[1]{\href{https://doi.org/#1}{\nolinkurl{https://doi.org/#1}}}

\newtheorem{theorem}{Theorem}[section]
\newtheorem{proposition}[theorem]{Proposition}
\newtheorem{corollary}[theorem]{Corollary}
\newtheorem{lemma}[theorem]{Lemma}
\newtheorem{definition}[theorem]{Definition}
\newtheorem{remark}[theorem]{Remark}

\newtheorem{example}[theorem]{Example}
\newtheorem*{explanatoryremark}{Remark}
\newtheorem*{explanatorycorollary}{Corollary}

\title{A Homological Decomposition for the Dimension and Dimensional Stability of Polynomial Spline Spaces over T-Meshes}
\author{Bingru Huang
\\[0.5ex]\small School of Mathematical Sciences, University of Science and Technology of China,\\ \small Hefei, 230026, People's Republic of China}
\date{}

\begin{document}
\maketitle

\noindent\textbf{Abstract.}

We study the dimension and dimensional stability of polynomial spline spaces of bi-degree $(m,m')$ with prescribed smoothness orders over planar T-meshes. The homological dimension formula writes the spline dimension as the sum of an Euler characteristic term and a correction term.  For a fixed ordered bi-degree, the Euler characteristic term is determined by the mesh structure and the prescribed smoothness orders.  The correction term can be written as a quotient of coefficient spaces attached to maximal interior segments (MISs).  We prove a weighted deletion theorem: when the available vertex relations generate the coefficient space of an MIS, its summand can be removed from the quotient without changing the correction term.  This operation changes neither the T-mesh nor its chain complexes.  Repeating the deletion leaves a weighted completely non-diagonalizable component (CNDC).  We prove that the weighted CNDC is independent of the order in which eligible MISs are removed and is the same for corresponding pairs in the structural class.  The correction term can therefore be represented using only the MISs in the weighted CNDC, while all relations from the original T-mesh are retained.  Dimensional stability is then equivalent to constancy of the dimension of the remaining relation space.  In particular, an empty weighted CNDC is sufficient for stability.  We also derive an upper bound for the remaining correction term and compare it with Mourrain's upper bound based on all MISs.


\section{Introduction}

Polynomial splines over T-meshes provide a flexible framework for local refinement.  The presence of T-junctions avoids the global propagation of tensor-product refinement, but it also makes the dimension problem more difficult.  The dimension is needed before a basis can be constructed and before the change of degrees of freedom under refinement can be understood.

For a general T-mesh, the spline dimension may depend on the positions of the mesh lines and not only on the incidence relations among its cells, edges, and vertices.  This phenomenon is usually called dimensional instability; see \citet{LiChen2011}, \citet{BerdinskyOhKimMourrain2012}, and \citet{LiWang2019}.  A dimension theory should therefore distinguish the data fixed by the mesh structure from the part that may still vary with the geometric realization.

Several methods have been developed for the dimension problem, including the B-net method \citep{DengChenFeng2006}, the smoothing cofactor method \citep{LiWangZhang2006,LiDeng2016,ZengDengLiDeng2015}, the minimal determining set method \citep{SchumakerWang2012,LaiSchumaker2007}, the space embedding method \citep{DengChenJin2013}, and homological algebra \citep{Billera1988,SchenckStillman1997,Mourrain2014}.  The homological method is particularly useful in the present setting because it gives
\[
\dim \mathcal S_{m,m'}^{\mathbf r}(\mathcal T)
=
\chi\!\left(\mathcal Q_{m,m'}^{\mathbf r}(\mathcal T^\circ)\right)
+
\dim H_0\!\left(\mathcal I_{m,m'}^{\mathbf r}(\mathcal T^\circ)\right).
\]
The first term is the Euler characteristic of the quotient complex.  For a fixed ordered bi-degree and fixed smoothness orders, it is determined by the numbers of cells, horizontal and vertical interior edges, and interior vertices.  The second term is the correction homology and records the remaining relations among the smoothness conditions.

Two different steps should be distinguished.  To study dimensional stability, we first compare geometric realizations with the same mesh structure and the same smoothness order on corresponding edges.  Their Euler characteristic terms are equal.  After fixing one realization $\mathcal T$, we simplify only the quotient representation of its correction homology.  This deletion is algebraic: it does not remove an edge from $\mathcal T$, merge cells, or change the complexes $\mathcal Q_{m,m'}^{\mathbf r}(\mathcal T^\circ)$ and $\mathcal I_{m,m'}^{\mathbf r}(\mathcal T^\circ)$.  The Euler characteristic term is therefore always computed on the original T-mesh.

Mourrain \citep{Mourrain2014} represents the correction homology as a quotient whose summands are indexed by maximal interior segments.  Each MIS is assigned a finite-dimensional polynomial coefficient space, and the vertices on the MISs give relations among these spaces.  This representation allows us to test whether the coefficient space of one MIS is already generated by the available vertex relations.  If so, that MIS summand can be removed without changing the quotient.

This idea is related to the decomposition of a T-connected component in the highest-order smoothness case \citep{HuangChen2024}, where a diagonalizable part is separated from a completely non-diagonalizable component.  Here we extend the deletion argument to arbitrary bi-degree and to smoothness distributions that may vary between maximal segments.  The contribution of each vertex is determined by the local degree and smoothness order.

We first define the structural class in which dimensional stability is studied.  We then prove a deletion theorem for this quotient representation of the correction term.  Repeating this deletion gives the weighted CNDC.  The correction homology can be represented using the MISs that remain, with all relations from the original T-mesh retained.  The stability problem is thereby reduced to the dimension of the remaining relation space.

The main contributions are as follows.
\begin{enumerate}[label=\textup{(\roman*)},leftmargin=*]
\item We define the structural class by requiring the same mesh structure and the same smoothness order on corresponding edges.  We prove that the Euler characteristic term is constant on this class.  We also show that the later deletion process, which is performed on a fixed T-mesh, does not change this term.

\item We give a quotient representation of the correction homology whose summands are indexed by MISs.  For any subset of MISs, we define the corresponding correction space by retaining every relation induced by the original T-mesh.

\item We prove a weighted deletion theorem.  If the relative weight of an MIS reaches the threshold determined by its direction, its coefficient summand can be removed without changing the correction space.

\item We give an iterative deletion algorithm that does not require a total order in advance.  Its final remaining set is the weighted CNDC.  We prove that this set is independent of the order in which eligible MISs are removed and is the same for corresponding pairs in the structural class.

\item We rewrite the dimension formula using the weighted CNDC and reduce dimensional stability to constancy of the dimension of the remaining relation space.  An empty weighted CNDC is sufficient for stability.  We also derive an upper bound for the remaining correction term.  For an order obtained from the deletion history, this bound has the same value as Mourrain's corresponding bound over all MISs, because the deleted MISs contribute zero.
\end{enumerate}

A nonempty weighted CNDC does not by itself imply dimensional instability.  It identifies the MISs that may contribute to the correction term, but the value of that term also depends on the dimension of the remaining relation space.  We include a highest-order smoothness example from \citet{LiChen2011}, also discussed in \citet{HuangChen2024}, in which the weighted CNDC is unchanged while the dimension of the remaining relation space and the spline dimension vary with the line coordinates.

The rest of the paper is organized as follows.  Section~\ref{sec:preliminaries} recalls T-meshes, smoothness distributions, spline spaces, and the homological complexes, and then defines the structural class used for comparison.  Section~\ref{sec:decomposition} proves the deletion theorem, constructs the weighted CNDC, and studies the remaining relation space, dimensional stability, and the upper bound.  Section~\ref{sec:conclusion} concludes the paper.

\section{Preliminaries}
\label{sec:preliminaries}

This section recalls the notation used in the paper.  We first define T-meshes, maximal segments, structural isomorphisms, and smoothness distributions.  We then introduce the spline space and the three chain complexes in the homological dimension formula.  The notation follows \citet{Mourrain2014} and the formulation for mixed smoothness in \citet{ToshniwalVillamizar2020}; related constructions appear in \citet{BraccoLycheManniRomanSpeleers2016}.

\subsection{T-meshes, smoothness distributions, and spline spaces}
\label{subsec:tmesh-spline}

Let $\Omega\subset \mathbb R^2$ be a finite union of closed axis-aligned rectangles whose
interiors are pairwise disjoint.  We assume throughout that $\Omega$ is simply connected
and that its interior $\Omega^\circ$ is connected.  We first specify the cells, edges,
and vertices of the T-mesh, since they appear in both the definition of the spline space
and the Euler characteristic term in the dimension formula.

\begin{definition}[T-mesh]
\label{def:t-mesh}
A planar T-mesh on $\Omega$ is a triple
\[
    \mathcal T=(\mathcal T_2,\mathcal T_1,\mathcal T_0)
\]
with the following properties:
\begin{itemize}
    \item $\mathcal T_2$ is the finite set of closed axis-aligned rectangles, called
    cells or $2$-cells;
    \item $\mathcal T_1=\mathcal T_1^h\cup \mathcal T_1^v$ is a finite set of closed
    horizontal and vertical segments contained in
    $\bigcup_{\sigma\in\mathcal T_2}\partial\sigma$, called edges or $1$-cells;
    \item $\mathcal T_0:=\bigcup_{\tau\in\mathcal T_1}\partial\tau$ is the finite set
    of endpoints of edges, called vertices or $0$-cells;
    \item for each $\sigma\in\mathcal T_2$, the boundary $\partial\sigma$ is a finite
    union of elements of $\mathcal T_1$;
    \item if $\sigma,\sigma'\in\mathcal T_2$ and $\sigma\neq\sigma'$, then
    $\sigma\cap\sigma'=\partial\sigma\cap\partial\sigma'$ is a finite union of elements
    of $\mathcal T_1\cup\mathcal T_0$;
    \item if $\tau,\tau'\in\mathcal T_1$ and $\tau\neq\tau'$, then
    $\tau\cap\tau'\subseteq\mathcal T_0$.
\end{itemize}
The elements of $\mathcal T_1$ meeting $\Omega^\circ$ are called interior edges and form
the set $\mathcal T_1^\circ$.  We write
\[
    \mathcal T_1^{\circ,h}:=\mathcal T_1^\circ\cap \mathcal T_1^h,
    \qquad
    \mathcal T_1^{\circ,v}:=\mathcal T_1^\circ\cap \mathcal T_1^v .
\]
The vertices lying in $\Omega^\circ$ are called interior vertices and form the set
$\mathcal T_0^\circ$.
\end{definition}

A segment of $\mathcal T$ is a connected union of edges lying on the same straight line.
For an interior edge $\tau\in\mathcal T_1^\circ$, we denote by $\rho(\tau)$ the maximal
segment consisting of interior edges, lying on the same line as $\tau$, and containing
$\tau$.  The set of all such maximal segments is denoted by $\operatorname{MS}(\mathcal T)$.
An element of $\operatorname{MS}(\mathcal T)$ is called a cross-cut if both of its endpoints lie on $\partial\Omega$, a ray if exactly one endpoint lies on $\partial\Omega$, and a maximal interior segment, abbreviated by MIS, if it does not meet $\partial\Omega$.  We write
\[
\operatorname{MS}(\mathcal T)
=
\operatorname{CC}(\mathcal T)
\,\dot\cup\,
\operatorname{Ray}(\mathcal T)
\,\dot\cup\,
\operatorname{MIS}(\mathcal T),
\qquad
\operatorname{MIS}(\mathcal T)
=
\operatorname{MIS}_h(\mathcal T)\cup\operatorname{MIS}_v(\mathcal T).
\]
A boundary maximal segment is a maximal connected union of boundary edges lying on the same straight line.  The set of boundary maximal segments is denoted by
$\operatorname{BMS}(\mathcal T)$.  When comparing T-meshes, we include both interior and boundary maximal segments and write
\[
\widehat{\operatorname{MS}}(\mathcal T)
:=
\operatorname{CC}(\mathcal T)
\,\dot\cup\,
\operatorname{Ray}(\mathcal T)
\,\dot\cup\,
\operatorname{MIS}(\mathcal T)
\,\dot\cup\,
\operatorname{BMS}(\mathcal T).
\]
The notation $\operatorname{MS}(\mathcal T)$ continues to refer only to maximal segments made of interior edges.  In the quotient representation introduced below, only the MISs give direct summands.

\paragraph{Assumptions on the T-mesh.}
We restrict attention to reduced T-meshes satisfying the following conditions.  No vertex is inserted merely to divide a straight edge into two collinear edges.  Every vertex lies on exactly one horizontal and one vertical member of $\widehat{\operatorname{MS}}(\mathcal T)$.  The edges of $\mathcal T_1$ are exactly the closed subsegments between consecutive vertices on these members.  In particular, an interior maximal segment and a boundary maximal segment of the same direction cannot meet collinearly at a common endpoint.  Under these assumptions, the maximal-segment types, their intersections, and the order of parallel maximal segments determine the one-skeleton without ambiguity.

\begin{figure}[t]
\centering
\begin{tikzpicture}[scale=0.78, every node/.style={font=\small}]
    % outer boundary
    \draw[thick] (0,0) rectangle (6,5);

    % interior edges
    \draw[thick] (2,0) -- (2,5);
    \draw[thick] (2,1) -- (6,1);
    \draw[thick] (2,4) -- (6,4);
    \draw[thick] (4,1) -- (4,4);
    \draw[thick] (2,2.5) -- (4,2.5);

    % marked vertices
    \fill (2,2.5) circle (2pt);
    \fill (4,2.5) circle (2pt);
    \fill (4,1) circle (2pt);
    \fill (4,4) circle (2pt);

    \node[left] at (2,2.5) {$v_1$};
    \node[below right] at (4,2.5) {$v_2$};
    \node[below right] at (4,1) {$v_3$};
    \node[above right] at (4,4) {$v_4$};
\end{tikzpicture}
\caption{A T-mesh example.}
\label{fig:basic-tmesh}
\end{figure}
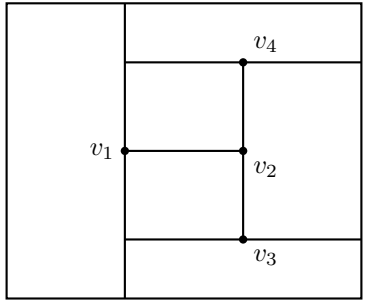

In Fig.~\ref{fig:basic-tmesh}, the horizontal segment $[v_1,v_2]$ and the vertical
segment $[v_3,v_2]\cup[v_2,v_4]$ are MISs.  The vertex $v_2$ lies on two MISs and is
therefore a multi-vertex relative to this connected collection of MISs, while
$v_1,v_3,v_4$ lie on only one MIS and are mono-vertices.

In the smoothing cofactor literature~\cite{ZengDengLiDeng2015,HuangChen2024}, maximal segments made of interior edges are classified as cross-cuts, rays, and T-large edges, also called T $l$-edges.  In the present homological terminology, a T-large edge is an MIS.  A connected collection of MISs together with their intersection vertices corresponds to a T-connected component.  Relative to such a collection, an interior vertex lying on exactly one MIS is called a mono-vertex, while an interior vertex lying on two MISs is called a multi-vertex.

To compare two geometric realizations, we need a precise rule for deciding when they have the same mesh structure.  The rule must preserve enough information to recover the vertices, edges, and cells, because the Euler characteristic term will later be compared across this class.  For this reason, the definition uses $\widehat{\operatorname{MS}}(\mathcal T)$ and includes the boundary maximal segments.

For $\rho\in\widehat{\operatorname{MS}}(\mathcal T)$, let
$\operatorname{dir}(\rho)\in\{h,v\}$ denote its direction.  If $\rho$ is
horizontal, let $\operatorname{cor}(\rho)$ be its constant $t$-coordinate; if
$\rho$ is vertical, let $\operatorname{cor}(\rho)$ be its constant
$s$-coordinate.  We write
\[
\widehat{\operatorname{MS}}(\mathcal T)
=
\widehat{\operatorname{MS}}_h(\mathcal T)
\cup
\widehat{\operatorname{MS}}_v(\mathcal T).
\]

\begin{definition}[Structural isomorphism]
\label{def:structural-isomorphism}
Let $\mathcal T$ and $\mathcal T'$ be two reduced T-meshes.  A bijection
\[
\alpha:
\widehat{\operatorname{MS}}(\mathcal T)
\longrightarrow
\widehat{\operatorname{MS}}(\mathcal T')
\]
is called a structural isomorphism if the following conditions hold:
\begin{enumerate}[label=\textup{(\roman*)},leftmargin=*]
\item The four maximal-segment types are preserved:
\[
\begin{aligned}
\alpha\bigl(\operatorname{CC}(\mathcal T)\bigr)
&=\operatorname{CC}(\mathcal T'),&
\alpha\bigl(\operatorname{Ray}(\mathcal T)\bigr)
&=\operatorname{Ray}(\mathcal T'),\\
\alpha\bigl(\operatorname{MIS}(\mathcal T)\bigr)
&=\operatorname{MIS}(\mathcal T'),&
\alpha\bigl(\operatorname{BMS}(\mathcal T)\bigr)
&=\operatorname{BMS}(\mathcal T').
\end{aligned}
\]

\item Incidence is preserved: for all distinct
$\rho,\eta\in\widehat{\operatorname{MS}}(\mathcal T)$,
\[
\rho\cap\eta\neq\varnothing
\quad\Longleftrightarrow\quad
\alpha(\rho)\cap\alpha(\eta)\neq\varnothing.
\]
When the intersection is nonempty, it is a vertex by the assumptions above, and the induced vertex correspondence is defined by
\[
\alpha_0(\rho\cap\eta)
:=
\alpha(\rho)\cap\alpha(\eta).
\]

\item Parallelism and relative coordinate order are preserved.  For
$\rho,\eta\in\widehat{\operatorname{MS}}(\mathcal T)$,
\[
\rho\parallel\eta
\quad\Longleftrightarrow\quad
\alpha(\rho)\parallel\alpha(\eta),
\]
and, whenever these segments are parallel,
\[
\operatorname{cor}(\rho)<\operatorname{cor}(\eta)
\quad\Longleftrightarrow\quad
\operatorname{cor}(\alpha(\rho))
<
\operatorname{cor}(\alpha(\eta)).
\]
\end{enumerate}
If such a map exists, then $\mathcal T$ and $\mathcal T'$ are structurally isomorphic.  The structural isomorphism is direction-preserving if it maps horizontal maximal segments to horizontal maximal segments and vertical maximal segments to vertical maximal segments.
\end{definition}

\begin{proposition}[Maps induced by a structural isomorphism]
\label{prop:induced-tmesh-isomorphism}
A structural isomorphism $\alpha$ either preserves the horizontal and vertical families globally or interchanges them globally.  Moreover, the induced vertex map $\alpha_0$ extends uniquely to bijections
\[
\alpha_1:\mathcal T_1\longrightarrow\mathcal T_1',
\qquad
\alpha_2:\mathcal T_2\longrightarrow\mathcal T_2'
\]
that preserve incidences and boundary and interior edge types, and such that
$\alpha_2$ maps the boundary cycle of each cell to the boundary cycle of its
image cell.  Thus
$(\alpha_2,\alpha_1,\alpha_0)$ is an isomorphism between the relative CW
complexes determined by the two T-meshes.
\end{proposition}

\begin{proof}
The one-skeleton of a T-mesh is connected.  Every edge lies on one member of
$\widehat{\operatorname{MS}}(\mathcal T)$, so the incidence graph of all maximal segments is connected.  This graph is bipartite: one part consists of
horizontal maximal segments and the other of vertical maximal segments.  A
connected bipartite graph has a unique bipartition up to interchanging its two
parts.  Hence $\alpha$ either preserves the two direction families or interchanges
them globally.

By the assumptions above, every vertex is the intersection of one horizontal and
one vertical member of $\widehat{\operatorname{MS}}(\mathcal T)$.  Condition~\textup{(ii)}
therefore defines a bijection $\alpha_0$ of vertices.  Fix a segment
$\rho\in\widehat{\operatorname{MS}}(\mathcal T)$.  The vertices on $\rho$ are ordered by the coordinate order of
the perpendicular maximal segments that meet $\rho$.  Condition~\textup{(iii)}
preserves this order.  Thus two consecutive vertices on $\rho$ are sent to two
consecutive vertices on $\alpha(\rho)$.  The closed subsegment between a
consecutive pair is an edge by the assumptions above, so this gives
a bijection $\alpha_1$ of edges.  Condition~\textup{(i)} preserves whether a maximal segment lies on the boundary and therefore preserves boundary and interior edges.

It remains to recover the cells.  At a vertex on a fixed horizontal maximal
segment, the order just described determines whether there is a horizontal
half-edge to the left, to the right, or to both sides; the same statement holds
vertically.  Hence the four possible local positions of incident half-edges are
preserved, up to the same global interchange of horizontal and vertical
directions.  Therefore $\alpha_1$ preserves the cyclic order of incident
half-edges at every vertex, again up to one global reversal.  The one-skeleton
together with these cyclic orders determines the face cycles of the planar
embedding.  Since boundary edges are mapped to boundary edges, the exterior face
is preserved.  The remaining face cycles are the boundaries of the rectangular
cells.  Mapping each such face to its corresponding face defines a bijection
$\alpha_2:\mathcal T_2\to\mathcal T_2'$ and proves the stated relative
CW-complex isomorphism.
\end{proof}

The definition records more than an unlabelled incidence graph.  It preserves the four types of maximal segments, their intersections with the boundary and with one another, and the relative order of parallel maximal segments.  Proposition~\ref{prop:induced-tmesh-isomorphism} shows that this information determines the vertices, edges, and cells of a reduced T-mesh.  When no confusion can arise, we use the same symbol $\alpha$ for the induced maps $\alpha_0$, $\alpha_1$, and $\alpha_2$.

The next definition specifies the continuity order across each edge.  We define the smoothness distribution separately because the structural class compares both the T-mesh and its prescribed smoothness.

\begin{definition}[Smoothness distribution {\cite{ToshniwalVillamizar2020}}]
\label{def:smoothness-distribution}
A smoothness distribution on $\mathcal T$ is a map
\[
    \mathbf r:\mathcal T_1\longrightarrow \mathbb Z_{\ge -1}:=\{-1,0,1,\ldots\}
\]
such that $\mathbf r(\tau)=-1$ for every boundary edge
$\tau\in\mathcal T_1\setminus\mathcal T_1^\circ$.  For an interior edge $\tau$,
the integer $\mathbf r(\tau)$ is the order of smoothness required across $\tau$.
\end{definition}

Throughout the remainder of the paper, we assume that the smoothness order is constant along each maximal segment, as in \citet{Mourrain2014}: for
every $\rho\in\operatorname{MS}(\mathcal T)$ and all interior edges
$\tau,\tau'\subset\rho$,
\[
    \mathbf r(\tau)=\mathbf r(\tau').
\]
The common value is denoted by $r(\rho)$.  Thus the distribution may be
non-uniform from one maximal segment to another, but it is constant along each
individual maximal segment.  This assumption allows us to factor the same power of the line equation from
all edge ideals on a given maximal segment in the deletion argument below.

To compare spline dimensions, we also require corresponding edges to have the same smoothness order.  We therefore compare pairs $(\mathcal T',\mathbf r')$, consisting of a T-mesh and its smoothness distribution.

\begin{definition}[Structural class]
\label{def:structural-class}
For reduced T-meshes, write
$\mathcal T\sim_{\mathrm I}\mathcal T'$ if there is a structural isomorphism
from $\mathcal T$ to $\mathcal T'$, and write
$\mathcal T\sim_{\mathrm I}^{+}\mathcal T'$ if such an isomorphism can be chosen
direction-preserving.  These are equivalence relations, and their equivalence
classes are denoted by $[\mathcal T]_{\mathrm I}$ and
$[\mathcal T]_{\mathrm I}^{+}$, respectively.

Now fix a smoothness distribution $\mathbf r$ on $\mathcal T$ with this property.  A pair $(\mathcal T',\mathbf r')$ belongs to the structural class
of $(\mathcal T,\mathbf r)$ if there is a direction-preserving structural
isomorphism $\alpha:\mathcal T\to\mathcal T'$ such that
\[
\mathbf r'\bigl(\alpha_1(\tau)\bigr)=\mathbf r(\tau)
\qquad
\text{for every }\tau\in\mathcal T_1^\circ,
\]
where $\alpha_1$ is the induced edge bijection in
Proposition~\ref{prop:induced-tmesh-isomorphism}.  Thus corresponding interior
edges have the same smoothness order.  We denote this class by
$[\mathcal T,\mathbf r]_{\mathrm I}^{+}$.

If two structural isomorphisms to the same geometric T-mesh lead to different
smoothness distributions, the resulting pairs are regarded as different elements
of the class.  Thus an element of $[\mathcal T,\mathbf r]_{\mathrm I}^{+}$ records
both the T-mesh and the smoothness distribution imposed on it.
\end{definition}

\begin{figure}[t]
\centering
\begin{tikzpicture}[scale=0.78, every node/.style={font=\small}]
    % outer boundary
    \draw[thick] (0,0) rectangle (5,6);

    % interior edges: obtained from Fig.~\ref{fig:basic-tmesh} by interchanging
    % the horizontal and vertical directions
    \draw[thick] (0,2) -- (5,2);
    \draw[thick] (1,2) -- (1,6);
    \draw[thick] (4,2) -- (4,6);
    \draw[thick] (1,4) -- (4,4);
    \draw[thick] (2.5,2) -- (2.5,4);

    % marked vertices
    \fill (2.5,2) circle (2pt);
    \fill (2.5,4) circle (2pt);
    \fill (1,4) circle (2pt);
    \fill (4,4) circle (2pt);

    \node[below right] at (2.5,2) {$v_1'$};
    \node[above left] at (2.5,4) {$v_2'$};
    \node[left] at (1,4) {$v_3'$};
    \node[right] at (4,4) {$v_4'$};
\end{tikzpicture}
\caption{A T-mesh structurally isomorphic to the T-mesh in
Fig.~\ref{fig:basic-tmesh}.}
\label{fig:structural-isomorphic-example}
\end{figure}
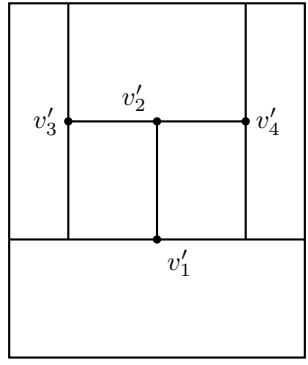

For example, let $\mathcal T$ be the T-mesh in Fig.~\ref{fig:basic-tmesh}, and
let $\mathcal T'$ be the T-mesh in Fig.~\ref{fig:structural-isomorphic-example}.
The correspondence sends each marked vertex $v_i$ to $v_i'$, sends each interior maximal segment to the corresponding interior maximal segment, and sends the four boundary maximal segments to their corresponding images.  It preserves maximal-segment types, incidences, and the coordinate order of parallel maximal segments.  Hence
$\mathcal T'\in[\mathcal T]_{\mathrm I}$.  The source mesh has a vertical
cross-cut, whereas the target mesh has a horizontal cross-cut.  Hence no
direction-preserving structural isomorphism exists, and
$\mathcal T'\notin[\mathcal T]_{\mathrm I}^{+}$.  After interchanging the two
variables, a spline space of ordered bi-degree $(m,m')$ on the first mesh
corresponds to one of ordered bi-degree $(m',m)$ on the second mesh.

For $(\mathcal T',\mathbf r')\in[\mathcal T,\mathbf r]_{\mathrm I}^{+}$,
the structural isomorphism identifies corresponding maximal segments, vertices,
edges, and cells.  It also assigns the same smoothness order to corresponding
interior edges.  The edge ideals, vertex ideals, and chain complexes can therefore
be compared term by term.  The spline space on the second T-mesh is
$\mathcal S_{m,m'}^{\mathbf r'}(\mathcal T')$.

A direction-interchanging structural isomorphism gives a separate symmetry.  After simultaneously interchanging $s$ and $t$, $m$ and $m'$, and the horizontal and vertical smoothness orders, it identifies the corresponding constructions on the two meshes.  Such maps belong to $[\mathcal T]_{\mathrm I}$, but they are not used in $[\mathcal T,\mathbf r]_{\mathrm I}^{+}$ when the ordered bi-degree is fixed.

For an interior vertex $\gamma\in\mathcal T_0^\circ$, let
$\rho_h(\gamma)$ and $\rho_v(\gamma)$ be the horizontal and vertical maximal
segments through $\gamma$, respectively.  Because $\mathbf r$ is constant along each maximal segment, we define
\[
    r_h(\gamma):=r\bigl(\rho_v(\gamma)\bigr),
    \qquad
    r_v(\gamma):=r\bigl(\rho_h(\gamma)\bigr).
\]
Thus $r_h(\gamma)$ records the smoothness imposed by the vertical maximal segment
through $\gamma$, and $r_v(\gamma)$ records the smoothness imposed by the
horizontal maximal segment through $\gamma$.  These values agree with
$\mathbf r(\tau)$ for every incident edge $\tau$ of the corresponding direction.

Let $R:=\mathbb R[s,t]$ be the polynomial ring in two variables.  For integers $m,m'\ge 0$, denote by
\[
    R_{m,m'}
    :=
    \operatorname{span}_{\mathbb R}\{s^i t^j:0\le i\le m,\ 0\le j\le m'\}
\]
the vector space of polynomials of bi-degree at most $(m,m')$.  We use the convention $R_{a,b}=\{0\}$ whenever $a<0$ or $b<0$, and
$\mathbb R[u]_{\le n}=\{0\}$ whenever $n<0$.

With the T-mesh and its smoothness distribution fixed, we can now define the spline space studied in this paper.

\begin{definition}[Polynomial spline space {\cite{ToshniwalVillamizar2020}}]
\label{def:polynomial-spline-space}
Let $\mathcal T$ be a T-mesh, let $(m,m')$ be a bi-degree, and let $\mathbf r$ be a
smoothness distribution.  The polynomial spline space of bi-degree $(m,m')$ and
smoothness distribution $\mathbf r$ over $\mathcal T$ is
\[
    \mathcal S_{m,m'}^{\mathbf r}(\mathcal T)
    :=
    \left\{
    f:\Omega\to\mathbb R
    \ \middle|\
    \begin{array}{l}
    f|_{\sigma}\in R_{m,m'}
    \text{ for every }\sigma\in\mathcal T_2,\\
    f\text{ is }C^{\mathbf r(\tau)}
    \text{ across every }\tau\in\mathcal T_1^\circ
    \end{array}
    \right\}.
\]
The boundary edges impose no smoothness condition.  We use the standard convention that $C^{-1}$ means that no matching condition is imposed; thus an interior edge may also be assigned the value $-1$.
\end{definition}

The following algebraic characterization of smoothness will be used repeatedly.

\begin{lemma}[Edge smoothness criterion {\cite{ChuiWang1983a,Billera1988}}]
\label{lem:edge-smoothness}
Let $\sigma,\sigma'\in\mathcal T_2$ be two cells sharing an interior edge
$\tau=\sigma\cap\sigma'$, and let $\ell_\tau\in R$ be a nonzero affine linear polynomial
vanishing on the line supporting $\tau$.  Suppose that a piecewise polynomial function is
given by $p$ on $\sigma$ and by $p'$ on $\sigma'$.  Then this function is
$C^r$ across $\tau$ if and only if
\[
    p-p'\in \left\langle \ell_\tau^{\,r+1}\right\rangle .
\]
\end{lemma}

\subsection{CW chain complexes}
\label{subsec:cw-chain-complexes}

We now introduce three CW chain complexes associated with the interior T-mesh
$\mathcal T^\circ$.  They are built from formal symbols $[\sigma]$, $[\tau]$, and
$[\gamma]$ attached respectively to cells, interior edges, and interior vertices.

For an interior edge $\tau\in\mathcal T_1^\circ$, define
\[
    I_\tau^{\mathbf r}
    :=
    \left\langle \ell_\tau^{\,\mathbf r(\tau)+1}\right\rangle\cap R_{m,m'} .
\]
For an interior vertex $\gamma\in\mathcal T_0^\circ$, define
\[
    I_\gamma^{\mathbf r}
    :=
    \sum_{\substack{\tau\ni\gamma\\ \tau\in\mathcal T_1^\circ}}
    I_\tau^{\mathbf r}
    \subseteq R_{m,m'} .
\]
Equivalently, $I_\gamma^{\mathbf r}$ is generated, in bi-degree $(m,m')$, by the
powers of the linear forms corresponding to the vertical and horizontal edges meeting
at $\gamma$.

The first complex is the relative CW chain complex with polynomial coefficients:
\[
\mathcal C_{m,m'}(\mathcal T^\circ):
\quad
0
\longrightarrow
\bigoplus_{\sigma\in\mathcal T_2}[\sigma]R_{m,m'}
\xrightarrow{\partial_2}
\bigoplus_{\tau\in\mathcal T_1^\circ}[\tau]R_{m,m'}
\xrightarrow{\partial_1}
\bigoplus_{\gamma\in\mathcal T_0^\circ}[\gamma]R_{m,m'}
\longrightarrow
0 .
\]
Choose an orientation for every interior edge $\tau\in\mathcal T_1^\circ$.  If
$\tau=[\gamma_-,\gamma_+]$ is oriented from $\gamma_-$ to $\gamma_+$, then
\[
    \partial_1([\tau]p)
    =
    [\gamma_+]p-[\gamma_-]p,
    \qquad p\in R_{m,m'}.
\]
Here $[\gamma]p$ is set equal to zero if $\gamma$ is a boundary vertex.  Each cell
$\sigma\in\mathcal T_2$ is oriented counterclockwise.  For $p\in R_{m,m'}$, define
\[
    \partial_2([\sigma]p)
    =
    \sum_{\tau\subset\partial\sigma\cap\mathcal T_1^\circ}
    \varepsilon_{\sigma,\tau}[\tau]p,
\]
where $\varepsilon_{\sigma,\tau}=1$ if the orientation induced by $\partial\sigma$
agrees with the chosen orientation of $\tau$, and
$\varepsilon_{\sigma,\tau}=-1$ otherwise.  With these definitions,
$\partial_1\circ\partial_2=0$.

The second complex is the ideal subcomplex of the relative CW chain complex determined by the smoothness distribution:
\[
\mathcal I_{m,m'}^{\mathbf r}(\mathcal T^\circ):
\quad
0
\longrightarrow
0
\longrightarrow
\bigoplus_{\tau\in\mathcal T_1^\circ}[\tau]I_\tau^{\mathbf r}
\xrightarrow{\widehat{\partial}_1}
\bigoplus_{\gamma\in\mathcal T_0^\circ}[\gamma]I_\gamma^{\mathbf r}
\longrightarrow
0 .
\]
The map $\widehat{\partial}_1$ is the restriction of $\partial_1$.  This is well defined
because $I_\tau^{\mathbf r}\subseteq I_\gamma^{\mathbf r}$ whenever $\gamma$ is an
endpoint of $\tau$.

The third complex is the quotient CW chain complex:
\[
\mathcal Q_{m,m'}^{\mathbf r}(\mathcal T^\circ):
\quad
0
\longrightarrow
\bigoplus_{\sigma\in\mathcal T_2}[\sigma]R_{m,m'}
\xrightarrow{\overline{\partial}_2}
\bigoplus_{\tau\in\mathcal T_1^\circ}
[\tau]\left(R_{m,m'}/I_\tau^{\mathbf r}\right)
\xrightarrow{\overline{\partial}_1}
\bigoplus_{\gamma\in\mathcal T_0^\circ}
[\gamma]\left(R_{m,m'}/I_\gamma^{\mathbf r}\right)
\longrightarrow
0 .
\]
The maps $\overline{\partial}_2$ and $\overline{\partial}_1$ are induced by
$\partial_2$ and $\partial_1$ after passing to quotients.

The three CW chain complexes fit into the following short exact sequence of chain complexes:
{\footnotesize
\[
\begin{array}{rccccccccc}
\mathcal I_{m,m'}^{\mathbf r}(\mathcal T^\circ):
&
0 & \longrightarrow & 0
& \longrightarrow &
\displaystyle\bigoplus_{\tau\in\mathcal T_1^\circ}[\tau]I_\tau^{\mathbf r}
& \xrightarrow{\widehat{\partial}_1} &
\displaystyle\bigoplus_{\gamma\in\mathcal T_0^\circ}[\gamma]I_\gamma^{\mathbf r}
& \longrightarrow & 0
\\[2mm]
&&& \downarrow && \downarrow && \downarrow &&
\\[-1mm]
\mathcal C_{m,m'}(\mathcal T^\circ):
&
0 & \longrightarrow &
\displaystyle\bigoplus_{\sigma\in\mathcal T_2}[\sigma]R_{m,m'}
& \xrightarrow{\partial_2} &
\displaystyle\bigoplus_{\tau\in\mathcal T_1^\circ}[\tau]R_{m,m'}
& \xrightarrow{\partial_1} &
\displaystyle\bigoplus_{\gamma\in\mathcal T_0^\circ}[\gamma]R_{m,m'}
& \longrightarrow & 0
\\[2mm]
&&& \downarrow && \downarrow && \downarrow &&
\\[-1mm]
\mathcal Q_{m,m'}^{\mathbf r}(\mathcal T^\circ):
&
0 & \longrightarrow &
\displaystyle\bigoplus_{\sigma\in\mathcal T_2}[\sigma]R_{m,m'}
& \xrightarrow{\overline{\partial}_2} &
\displaystyle\bigoplus_{\tau\in\mathcal T_1^\circ}
[\tau]\left(R_{m,m'}/I_\tau^{\mathbf r}\right)
& \xrightarrow{\overline{\partial}_1} &
\displaystyle\bigoplus_{\gamma\in\mathcal T_0^\circ}
[\gamma]\left(R_{m,m'}/I_\gamma^{\mathbf r}\right)
& \longrightarrow & 0 .
\end{array}
\]
}
The vertical maps from $\mathcal I_{m,m'}^{\mathbf r}$ to
$\mathcal C_{m,m'}$ are inclusions, and those from $\mathcal C_{m,m'}$ to
$\mathcal Q_{m,m'}^{\mathbf r}$ are quotient maps.

The quotient complex was constructed so that the top cycles are exactly the
piecewise polynomials satisfying the prescribed edge smoothness conditions.
Using Lemma~\ref{lem:edge-smoothness}, this gives the following homological
characterization of the spline space.

\begin{theorem}[Spline space as top homology {\cite{Mourrain2014}}]
\label{thm:spline-homology}
For a simply connected T-mesh $\mathcal T$, there is a natural vector space
isomorphism
\[
    \mathcal S_{m,m'}^{\mathbf r}(\mathcal T)
    \cong
    H_2\!\left(\mathcal Q_{m,m'}^{\mathbf r}(\mathcal T^\circ)\right).
\]
\end{theorem}

The homology of the middle complex is purely topological.  Since $\Omega$ is
simply connected and $\Omega^\circ$ is connected, the relative CW chain complex
satisfies
\[
    H_0\!\left(\mathcal C_{m,m'}(\mathcal T^\circ)\right)=0,
    \qquad
    H_1\!\left(\mathcal C_{m,m'}(\mathcal T^\circ)\right)=0.
\]
The long exact homology sequence induced by the short exact sequence of
complexes gives
\[
    H_0\!\left(\mathcal Q_{m,m'}^{\mathbf r}(\mathcal T^\circ)\right)=0,
    \qquad
    H_1\!\left(\mathcal Q_{m,m'}^{\mathbf r}(\mathcal T^\circ)\right)
    \cong
    H_0\!\left(\mathcal I_{m,m'}^{\mathbf r}(\mathcal T^\circ)\right).
\]
Therefore, by the Euler characteristic identity for the quotient complex, one
obtains the following dimension formula.

\begin{theorem}[Homological dimension formula {\cite{Mourrain2014}}]
\label{thm:homological-dimension-formula}
The dimension of the spline space is
\[
    \dim \mathcal S_{m,m'}^{\mathbf r}(\mathcal T)
    =
    \chi\!\left(\mathcal Q_{m,m'}^{\mathbf r}(\mathcal T^\circ)\right)
    +
    \dim H_0\!\left(\mathcal I_{m,m'}^{\mathbf r}(\mathcal T^\circ)\right),
\]
where
\[
\begin{aligned}
\chi\!\left(\mathcal Q_{m,m'}^{\mathbf r}(\mathcal T^\circ)\right)
={}&
|\mathcal T_2|(m+1)(m'+1)
\\
&-
\sum_{\tau\in\mathcal T_1^{\circ,h}}
(m+1)\bigl(\min\{\mathbf r(\tau),m'\}+1\bigr)
\\
&-
\sum_{\tau\in\mathcal T_1^{\circ,v}}
\bigl(\min\{\mathbf r(\tau),m\}+1\bigr)(m'+1)
\\
&+
\sum_{\gamma\in\mathcal T_0^\circ}
\bigl(\min\{r_h(\gamma),m\}+1\bigr)
\bigl(\min\{r_v(\gamma),m'\}+1\bigr).
\end{aligned}
\]
\end{theorem}

\begin{explanatoryremark}[The Euler characteristic term]
The quantity
\[
\chi\!\left(\mathcal Q_{m,m'}^{\mathbf r}(\mathcal T^\circ)\right)
\]
is the Euler characteristic of the quotient complex, not the ordinary Euler characteristic of the domain $\Omega$.  It is computed from the terms corresponding to cells, interior edges, and interior vertices in Theorem~\ref{thm:homological-dimension-formula}, together with the ordered bi-degree and the smoothness orders.  It is therefore independent of the numerical positions of the mesh lines when the mesh structure and the smoothness orders are fixed.  The deletion process in Section~\ref{sec:decomposition} changes only the quotient representation of the correction homology.  It does not change the T-mesh $\mathcal T$ or the quotient complex $\mathcal Q_{m,m'}^{\mathbf r}(\mathcal T^\circ)$, and hence does not change this Euler characteristic term.
\end{explanatoryremark}

\begin{corollary}[Vanishing correction homology {\cite{Mourrain2014}}]
\label{cor:vanishing-correction}
If
\[
    H_0\!\left(\mathcal I_{m,m'}^{\mathbf r}(\mathcal T^\circ)\right)=0,
\]
then
\[
    \dim \mathcal S_{m,m'}^{\mathbf r}(\mathcal T)
    =
    \chi\!\left(\mathcal Q_{m,m'}^{\mathbf r}(\mathcal T^\circ)\right).
\]
In this case the dimension is determined by the Euler characteristic of the
quotient complex.
\end{corollary}

\section{Deleting MIS summands and reducing the correction term}
\label{sec:decomposition}

This section studies the correction term in the dimension formula.  We first define dimensional stability within the structural class and prove that the Euler characteristic term is constant on this class.  We then use a quotient representation indexed by the MISs to remove summands generated by vertex relations.  The MISs that remain form the weighted CNDC.  Finally, we express the correction term through the remaining relation space and derive a stability criterion and an upper bound.

Throughout the section, the reference pair $(\mathcal T,\mathbf r)$ and the ordered bi-degree $(m,m')$ are fixed, and $\mathbf r$ is constant along each maximal segment.  Comparisons between different realizations are made only in the class $[\mathcal T,\mathbf r]_{\mathrm I}^{+}$.

For the deletion process, one realization is fixed.  At every stage, the cells, edges, vertices, smoothness ideals, and chain complexes of $\mathcal T$ are unchanged.  Only the subset $\mathcal A\subseteq\operatorname{MIS}(\mathcal T)$ indexing the retained summands is updated.  In particular, the Euler characteristic term $\chi(\mathcal Q_{m,m'}^{\mathbf r}(\mathcal T^\circ))$ is never recomputed during the deletion algorithm.

\subsection{Dimensional stability within a structural class}
\label{subsec:dimension-classes}

To measure changes of dimension within the structural class, we compare every
pair with the fixed reference pair $(\mathcal T,\mathbf r)$.

\begin{definition}[Dimension classes]
\label{def:dimension-classes}
For $k\in\mathbb Z$, define
\[
[k]_{(\mathcal T,\mathbf r)}
:=
\left\{
(\mathcal T',\mathbf r')\in[\mathcal T,\mathbf r]_{\mathrm I}^{+}:
\dim \mathcal S_{m,m'}^{\mathbf r'}(\mathcal T')
-
\dim \mathcal S_{m,m'}^{\mathbf r}(\mathcal T)
=
k
\right\}.
\]
If this set is nonempty, it is called the $k$-th dimension class relative to the
reference pair $(\mathcal T,\mathbf r)$.
\end{definition}

The zero class contains the realizations with the same spline dimension as the
reference pair.  This leads to the stability definition used below.

\begin{definition}[Dimensional stability]
\label{def:dimensional-stability}
For the fixed ordered bi-degree $(m,m')$, the reference pair
$(\mathcal T,\mathbf r)$ is called dimensionally stable if
\[
[\mathcal T,\mathbf r]_{\mathrm I}^{+}
=
[0]_{(\mathcal T,\mathbf r)}.
\]
Equivalently,
\[
\dim \mathcal S_{m,m'}^{\mathbf r'}(\mathcal T')
=
\dim \mathcal S_{m,m'}^{\mathbf r}(\mathcal T)
\]
for every
$(\mathcal T',\mathbf r')\in[\mathcal T,\mathbf r]_{\mathrm I}^{+}$.
Otherwise the reference pair is called dimensionally unstable.
\end{definition}

The definition allows any smoothness distribution that is constant along each
maximal segment; it is not restricted to highest-order smoothness.  A structural isomorphism fixes
the incidences and relative order of corresponding maximal segments, while the
definition of $[\mathcal T,\mathbf r]_{\mathrm I}^{+}$ fixes the smoothness order
on each corresponding edge.  A change of dimension within this class is therefore
due to the geometric realization, not to a change in the prescribed smoothness.

By Theorem~\ref{thm:homological-dimension-formula}, for every
$(\mathcal T',\mathbf r')\in[\mathcal T,\mathbf r]_{\mathrm I}^{+}$,
\[
\dim \mathcal S_{m,m'}^{\mathbf r'}(\mathcal T')
=
\chi\!\left(
\mathcal Q_{m,m'}^{\mathbf r'}((\mathcal T')^\circ)
\right)
+
\dim H_0\!\left(
\mathcal I_{m,m'}^{\mathbf r'}((\mathcal T')^\circ)
\right).
\]

The next lemma explains why only the correction homology needs to be compared.

\begin{lemma}[Euler characteristic invariance]
\label{lem:euler-characteristic-structural-invariance}
Let
$(\mathcal T_1,\mathbf r_1),(\mathcal T_2,\mathbf r_2)
\in[\mathcal T,\mathbf r]_{\mathrm I}^{+}$.  Then
\[
\chi\!\left(
\mathcal Q_{m,m'}^{\mathbf r_1}(\mathcal T_1^\circ)
\right)
=
\chi\!\left(
\mathcal Q_{m,m'}^{\mathbf r_2}(\mathcal T_2^\circ)
\right).
\]
\end{lemma}

\begin{proof}
Choose direction-preserving structural isomorphisms
$\alpha_i:\mathcal T\to\mathcal T_i$, $i=1,2$, that satisfy the smoothness
condition in Definition~\ref{def:structural-class}.  Then
$\alpha_2\circ\alpha_1^{-1}$ is a direction-preserving structural isomorphism
from $\mathcal T_1$ to $\mathcal T_2$ and preserves the smoothness orders on
corresponding edges.  By
Proposition~\ref{prop:induced-tmesh-isomorphism}, corresponding cells, horizontal
interior edges, vertical interior edges, and interior vertices are in bijection.
The smoothness orders agree on corresponding edges, and the pairs
$\bigl(r_h(\gamma),r_v(\gamma)\bigr)$ agree at corresponding vertices.  Each
summand in the Euler characteristic formula of
Theorem~\ref{thm:homological-dimension-formula} is therefore the same on the two
realizations.
\end{proof}

Consequently, the difference between the spline dimensions of two pairs in the structural class is exactly the difference between their correction terms.  Dimensional stability can therefore be studied by considering only the correction term.  This comparison of different realizations is separate from the deletion process carried out below on each fixed T-mesh.

\subsection{Maximal-segment relations and deletion of one summand}
\label{subsec:relative-weights-deletion}

The deletion theorem uses a quotient whose summands are indexed by maximal interior segments.  We recall this representation from \citet{Mourrain2014}.

For an MIS $\rho$, let $\ell_\rho\in R$ be a nonzero affine linear polynomial
vanishing on the line supporting $\rho$.  If $\rho$ is horizontal, then
$\ell_\rho$ is a polynomial in the vertical variable; if $\rho$ is vertical, then
$\ell_\rho$ is a polynomial in the horizontal variable.  For a vertex
$\gamma=\rho_h\cap\rho_v$, where $\rho_h$ is horizontal and $\rho_v$ is
vertical, the two powers appearing at $\gamma$ are
\[
    \ell_{\rho_v}^{\,r_h(\gamma)+1},
    \qquad
    \ell_{\rho_h}^{\,r_v(\gamma)+1}.
\]
Here $r_h(\gamma)$ is the smoothness imposed by the vertical edges through
$\gamma$, and $r_v(\gamma)$ is the smoothness imposed by the horizontal edges
through $\gamma$.

\begin{proposition}[Quotient representation of the correction term {\cite{Mourrain2014}}]
\label{prop:mis-representation}
The correction homology can be written as the following quotient over the
maximal interior segments:
\[
    H_0\!\left(
    \mathcal I_{m,m'}^{\mathbf r}(\mathcal T^\circ)
    \right)
    \cong
    \left(
    \bigoplus_{\rho\in\operatorname{MIS}(\mathcal T)}
    [\rho]\mathcal B_\rho
    \right)
    /
    \mathcal R(\mathcal T).
\]
Because $\mathbf r$ is constant along each maximal segment, the coefficient space attached to an MIS $\rho$ is
\[
    \mathcal B_\rho
    =
    \begin{cases}
    R_{m,m'-r(\rho)-1},
    &
    \rho\in\operatorname{MIS}_h(\mathcal T),\\[1mm]
    R_{m-r(\rho)-1,m'},
    &
    \rho\in\operatorname{MIS}_v(\mathcal T).
    \end{cases}
\]
Indeed, all edge ideals on $\rho$ contain the same factor
$\ell_\rho^{\,r(\rho)+1}$.  After removing this common factor, the remaining
coefficient space is $\mathcal B_\rho$.  This is where we use the assumption that the smoothness order is constant along
each maximal segment.

The subspace $\mathcal R(\mathcal T)$ is generated by the vertex relations.  If
\[
    \gamma=\rho_h\cap\rho_v
\]
is an interior vertex lying on a horizontal MIS $\rho_h$ and a vertical MIS
$\rho_v$, then the corresponding relation has the form
\[
    [\rho_v]\ell_{\rho_h}^{\,r(\rho_h)+1}q
    -
    [\rho_h]\ell_{\rho_v}^{\,r(\rho_v)+1}q,
\]
where
\[
    q\in R_{m-r(\rho_v)-1,\,m'-r(\rho_h)-1}.
\]
Equivalently, using the vertex notation introduced above, this is
\[
    [\rho_v]\ell_{\rho_h}^{\,r_v(\gamma)+1}q
    -
    [\rho_h]\ell_{\rho_v}^{\,r_h(\gamma)+1}q.
\]
The relation records the fact that the two edge ideals meeting at $\gamma$
generate the vertex ideal
\[
    I_\gamma^{\mathbf r}
    =
    \left\langle
    \ell_{\rho_v}^{\,r(\rho_v)+1},
    \ell_{\rho_h}^{\,r(\rho_h)+1}
    \right\rangle
    \cap R_{m,m'} .
\]

If $\gamma$ is a mono-vertex of an MIS $\rho$, the perpendicular maximal
segment through $\gamma$ meets the boundary and does not give an MIS summand in
the quotient.  The corresponding vertex relation therefore
involves only the summand $[\rho]\mathcal B_\rho$.  These
one-term relations are included in $\mathcal R(\mathcal T)$ and are used together
with the two-term relations at intersections of two MISs.

\end{proposition}

After some MIS summands have been removed, every relation among the remaining
summands must still be retained, including relations obtained by eliminating the
removed summands.  We therefore use the following definition.

For a subset $\mathcal A\subseteq\operatorname{MIS}(\mathcal T)$, set
\[
    \mathcal M(\mathcal A)
    :=
    \bigoplus_{\rho\in\mathcal A}[\rho]\mathcal B_\rho
\]
and define the relation space on $\mathcal A$ by
\[
    \mathcal R_{\mathcal A}(\mathcal T)
    :=
    \mathcal M(\mathcal A)\cap\mathcal R(\mathcal T).
\]
The correction space associated with $\mathcal A$ is
\[
    \mathcal H_{m,m'}^{\mathbf r}(\mathcal A;\mathcal T)
    :=
    \mathcal M(\mathcal A)/\mathcal R_{\mathcal A}(\mathcal T).
\]
The intersection with the full relation space is essential.  It can be larger
than the span of the original vertex relations that involve only summands in
$\mathcal A$.  By combining relations and eliminating summands outside
$\mathcal A$, one may obtain additional relations among the retained summands.
The space $\mathcal R_{\mathcal A}(\mathcal T)$ contains all of these relations.

For the full MIS set,
\[
    \mathcal H_{m,m'}^{\mathbf r}
    \bigl(\operatorname{MIS}(\mathcal T);\mathcal T\bigr)
    \cong
    H_0\!\left(
    \mathcal I_{m,m'}^{\mathbf r}(\mathcal T^\circ)
    \right).
\]

For later use, we also give a matrix description of this correction space.  Fix a basis
of $\mathcal M(\mathcal A)$ and choose any finite spanning set of
$\mathcal R_{\mathcal A}(\mathcal T)$.  Let $K_{\mathcal A}(\mathcal T)$ be the
matrix whose columns are the coordinates of that spanning set.  Then
\[
    \operatorname{rank}K_{\mathcal A}(\mathcal T)
    =
    \dim\mathcal R_{\mathcal A}(\mathcal T)
\]
and hence
\[
    \dim\mathcal H_{m,m'}^{\mathbf r}(\mathcal A;\mathcal T)
    =
    \dim\mathcal M(\mathcal A)
    -
    \operatorname{rank}K_{\mathcal A}(\mathcal T).
\]
The matrix is not unique, but its rank is.  We introduce it only to express the
rank condition used below.  In particular, when $\mathcal A$ is the weighted CNDC, the matrix
$K_{\mathcal A}(\mathcal T)$ may still involve coefficients arising from MISs
that were removed, because the relation space is defined using
$\mathcal R(\mathcal T)$.

\begin{explanatoryremark}[Retaining a subset of MIS summands does not change the T-mesh]
A subset $\mathcal A\subseteq\operatorname{MIS}(\mathcal T)$ only specifies which MIS summands are retained in the quotient representation.  It is not a new T-mesh and does not define a new quotient complex.  In particular, no cell is merged and no edge or vertex is removed from $\mathcal T$.  The semicolon in
\[
\mathcal H_{m,m'}^{\mathbf r}(\mathcal A;\mathcal T)
\]
records that the correction space is still formed using the full relation space of the original T-mesh.  Therefore changing $\mathcal A$ has no effect on $\chi(\mathcal Q_{m,m'}^{\mathbf r}(\mathcal T^\circ))$.
\end{explanatoryremark}

To decide whether one MIS summand can be removed, we must specify which vertex relations are available at that step.  We therefore choose a total order of the full MIS set.  An order on a proper subset is not sufficient, because an intersection with an omitted MIS need not give a relation among the currently retained summands.

\begin{definition}[Relative vertices and relative weight]
\label{def:relative-weight}
Let
\[
\iota:\quad
\rho_1\succ\rho_2\succ\cdots\succ\rho_t
\]
be a total order of
\[
\operatorname{MIS}(\mathcal T)
=
\{\rho_1,\ldots,\rho_t\}.
\]
The relative vertex set of $\rho_i$ with respect to $\iota$ is
\[
\Gamma_{\iota}(\rho_i)
:=
\left\{
\gamma\in\rho_i:
\gamma\notin\rho_j
\text{ for all }j<i
\right\}.
\]
Equivalently, $\Gamma_{\iota}(\rho_i)$ consists of the original mono-vertices of
$\rho_i$ together with the intersection vertices of $\rho_i$ with the MISs that
come after $\rho_i$ in the total order.

For an interior vertex $\gamma$, set
\[
\mu_h(\gamma):=\bigl(m-r_h(\gamma)\bigr)_+,
\qquad
\mu_v(\gamma):=\bigl(m'-r_v(\gamma)\bigr)_+,
\]
where $(a)_+=\max\{a,0\}$.  The subscript $h$ refers to the polynomial degree in
the horizontal variable and therefore uses the smoothness imposed by the
vertical maximal segment; the analogous convention applies to $\mu_v$.
The relative weight of $\rho_i$ is
\[
\omega_{\iota}(\rho_i)
=
\begin{cases}
\displaystyle
\sum_{\gamma\in\Gamma_{\iota}(\rho_i)}
\mu_h(\gamma),
&
\rho_i\in\operatorname{MIS}_h(\mathcal T),
\\[4mm]
\displaystyle
\sum_{\gamma\in\Gamma_{\iota}(\rho_i)}
\mu_v(\gamma),
&
\rho_i\in\operatorname{MIS}_v(\mathcal T).
\end{cases}
\]
The corresponding threshold is
\[
\theta(\rho_i)
=
\begin{cases}
m+1,&\rho_i\in\operatorname{MIS}_h(\mathcal T),\\
m'+1,&\rho_i\in\operatorname{MIS}_v(\mathcal T).
\end{cases}
\]
\end{definition}

\paragraph{Example.}
Consider the two MISs in Fig.~\ref{fig:basic-tmesh}.  Let
\[
    \rho_h=[v_1,v_2],
    \qquad
    \rho_v=[v_3,v_2]\cup[v_2,v_4].
\]
Assume for simplicity that $m=m'=d$ and that the smoothness is of highest order,
so that every vertex has weight one.  If the order is
\[
    \rho_h\succ\rho_v,
\]
then
\[
    \Gamma_{\iota}(\rho_h)=\{v_1,v_2\},
    \qquad
    \Gamma_{\iota}(\rho_v)=\{v_3,v_4\}.
\]
Thus the relative weights are
\[
    \omega_{\iota}(\rho_h)=2,
    \qquad
    \omega_{\iota}(\rho_v)=2.
\]
If the order is reversed,
\[
    \rho_v\succ\rho_h,
\]
then
\[
    \Gamma_{\iota}(\rho_v)=\{v_3,v_2,v_4\},
    \qquad
    \Gamma_{\iota}(\rho_h)=\{v_1\}.
\]
Thus
\[
    \omega_{\iota}(\rho_v)=3,
    \qquad
    \omega_{\iota}(\rho_h)=1.
\]
This example shows that the relative weight is not attached only to the MIS
itself; it also depends on the chosen order.

The proof of the deletion theorem uses the following one-dimensional fact.  We
state it separately so that the hypotheses used in the proof are explicit.

\begin{lemma}[Univariate generation formula {\cite{Mourrain2014}}]
\label{lem:univariate-generation}
Let $a_1,\ldots,a_q$ be distinct real numbers, let $d\ge0$, and let
$r_1,\ldots,r_q\ge -1$.  With the convention
$\mathbb R[u]_{\le n}=\{0\}$ for $n<0$,
\[
\dim\!\left(
\mathbb R[u]_{\le d}
/
\sum_{j=1}^{q}
(u-a_j)^{r_j+1}\mathbb R[u]_{\le d-r_j-1}
\right)
=
\left(
d+1-\sum_{j=1}^{q}(d-r_j)_+
\right)_+ .
\]
The statement also covers $r_j\ge d$, for which the corresponding summand is
zero.
\end{lemma}

This is the univariate generation formula used in the maximal-segment estimates
of \citet{Mourrain2014}.  In the proof below the points $a_j$ are the distinct
coordinates of the relative vertices on one MIS, and the integers $r_j$ are the
smoothness orders on the perpendicular maximal segments.

\begin{theorem}[Deletion of an MIS summand]
\label{thm:homological-deletion}
Let
\[
\iota:\rho_1\succ\rho_2\succ\cdots\succ\rho_t
\]
be a total order of $\operatorname{MIS}(\mathcal T)$, and set
\[
\mathcal A_i:=\{\rho_i,\rho_{i+1},\ldots,\rho_t\}.
\]
If
\[
\omega_{\iota}(\rho_i)\ge\theta(\rho_i),
\]
then the inclusion
$\mathcal M(\mathcal A_{i+1})\hookrightarrow\mathcal M(\mathcal A_i)$
induces an isomorphism
\[
\mathcal H_{m,m'}^{\mathbf r}(\mathcal A_{i+1};\mathcal T)
\cong
\mathcal H_{m,m'}^{\mathbf r}(\mathcal A_i;\mathcal T).
\]
Here $\mathcal A_{t+1}:=\varnothing$.  The theorem removes only the summand indexed by $\rho_i$ from the quotient representation; all relations induced among the remaining summands are retained.  The segment $\rho_i$ is not removed from the T-mesh itself.  Thus the cells, edges, vertices, quotient complex, and Euler characteristic term are unchanged.
\end{theorem}

\begin{proof}
We prove the horizontal case; the vertical case follows by interchanging the two
variables.  Assume that $\rho_i$ is horizontal and write
\[
    r_i:=r(\rho_i).
\]
Let
\[
    \mathcal V_i:=\mathcal M(\mathcal A_i),
    \qquad
    \mathcal W_i:=
    \mathcal M(\mathcal A_{i+1}),
    \qquad
    \mathcal R_i:=\mathcal V_i\cap\mathcal R(\mathcal T).
\]
Then
\[
    \mathcal V_i=[\rho_i]\mathcal B_{\rho_i}\oplus\mathcal W_i.
\]
We show that the projection of $\mathcal R_i$ onto the first summand is
surjective and then use this fact to identify the two correction spaces.

\begin{enumerate}[label=\textup{(\arabic*)},leftmargin=*]
    \item \emph{The coefficient space on $\rho_i$.}
    Since the smoothness is constant along $\rho_i$,
    \[
        \mathcal B_{\rho_i}
        =R_{m,m'-r_i-1}
        =\mathbb R[s]_{\le m}\otimes
        \mathbb R[t]_{\le m'-r_i-1}.
    \]
    Thus, after removing the common factor
    $\ell_{\rho_i}^{\,r_i+1}$, the polynomial factor along $\rho_i$ is the
    full univariate space $\mathbb R[s]_{\le m}$, while the transverse factor
    is fixed.

    Write
    \[
        \Gamma_{\iota}(\rho_i)
        =\{\gamma_1,\ldots,\gamma_q\},
        \qquad
        \gamma_j=(s_j,t_0).
    \]
    The values $s_1,\ldots,s_q$ are pairwise distinct.

    \item \emph{Projection of the vertex relations.}
    For each $\gamma_j$, let $\eta_j$ be the vertical maximal segment through
    $\gamma_j$ and put $r_j:=r(\eta_j)=r_h(\gamma_j)$.  The line supporting
    $\eta_j$ has equation $s-s_j=0$.  The relation at $\gamma_j$ has
    $\rho_i$-component
    \[
        [\rho_i](s-s_j)^{r_j+1}q,
        \qquad
        q\in R_{m-r_j-1,\,m'-r_i-1}.
    \]
    If $\gamma_j$ is an original mono-vertex, this is a one-term relation.  If
    $\gamma_j$ is the intersection of $\rho_i$ with an MIS occurring after
    $\rho_i$ in the order $\iota$, both segment summands belong to
    $\mathcal V_i$.  In either case the relation belongs to $\mathcal R_i$.
    Hence the projection of $\mathcal R_i$ onto
    $[\rho_i]\mathcal B_{\rho_i}$ contains
    \[
        [\rho_i]\left(
        \sum_{\gamma_j\in\Gamma_{\iota}(\rho_i)}
        (s-s_j)^{r_j+1}\mathbb R[s]_{\le m-r_j-1}
        \right)
        \otimes
        \mathbb R[t]_{\le m'-r_i-1}.
    \]

    \item \emph{Surjectivity in the $\rho_i$-summand.}
    Lemma~\ref{lem:univariate-generation} gives
    \[
    \begin{aligned}
    &\dim\!\left(
    \mathbb R[s]_{\le m}
    /
    \sum_{\gamma_j\in\Gamma_{\iota}(\rho_i)}
    (s-s_j)^{r_j+1}
    \mathbb R[s]_{\le m-r_j-1}
    \right)
    \\
    &\qquad=
    \left(
    m+1-
    \sum_{\gamma_j\in\Gamma_{\iota}(\rho_i)}
    (m-r_j)_+
    \right)_+
    =
    \bigl(m+1-\omega_{\iota}(\rho_i)\bigr)_+.
    \end{aligned}
    \]
    Since $\omega_{\iota}(\rho_i)\ge m+1$, this quotient is zero.  Tensoring
    with the transverse factor shows that the projection
    \[
        \pi_i:\mathcal R_i\longrightarrow
        [\rho_i]\mathcal B_{\rho_i}
    \]
    is surjective.

    \item \emph{Identification of the correction spaces.}
    The inclusion $\mathcal W_i\hookrightarrow\mathcal V_i$ induces a linear
    map
    \[
        \Phi_i:
        \mathcal W_i/
        \bigl(\mathcal W_i\cap\mathcal R(\mathcal T)\bigr)
        \longrightarrow
        \mathcal V_i/\mathcal R_i.
    \]
    Surjectivity of $\pi_i$ implies that every class in the target has a
    representative in $\mathcal W_i$, so $\Phi_i$ is surjective.  Its kernel is
    \[
        \mathcal W_i\cap\mathcal R_i
        =\mathcal W_i\cap\mathcal R(\mathcal T),
    \]
    which is exactly the denominator on the left.  Therefore $\Phi_i$ is an
    isomorphism.  By the definition of the correction spaces,
    \[
        \mathcal H_{m,m'}^{\mathbf r}(\mathcal A_i;\mathcal T)
        \cong
        \mathcal H_{m,m'}^{\mathbf r}
        (\mathcal A_{i+1};\mathcal T).
    \]
    This is exactly the isomorphism asserted in the theorem.
\end{enumerate}

\end{proof}

Repeated use of Theorem~\ref{thm:homological-deletion} gives the following consequence for a fixed T-mesh.

\begin{explanatorycorollary}[The Euler characteristic is unchanged during deletion]
Let
\[
\operatorname{MIS}(\mathcal T)=\mathcal A^{(0)}\supseteq\mathcal A^{(1)}\supseteq\cdots\supseteq\mathcal A^{(N)}
\]
be a sequence of MIS sets obtained by successive applications of Theorem~\ref{thm:homological-deletion}.  Then, for every $k=0,\ldots,N$,
\[
\dim \mathcal S_{m,m'}^{\mathbf r}(\mathcal T)
=
\chi\!\left(\mathcal Q_{m,m'}^{\mathbf r}(\mathcal T^\circ)\right)
+
\dim \mathcal H_{m,m'}^{\mathbf r}(\mathcal A^{(k)};\mathcal T).
\]
The Euler characteristic term is always the term associated with the original T-mesh $\mathcal T$; it does not vary with $k$.
\end{explanatorycorollary}

\begin{proof}
For $k=0$, the statement follows from Theorem~\ref{thm:homological-dimension-formula} and Proposition~\ref{prop:mis-representation}.  Each subsequent deletion gives an isomorphism between two correction spaces by Theorem~\ref{thm:homological-deletion}.  Thus only the quotient representation of the correction term changes.  The quotient complex $\mathcal Q_{m,m'}^{\mathbf r}(\mathcal T^\circ)$ is fixed, and the displayed formula follows for every $k$.
\end{proof}

We next compare the correction terms obtained after deletion on two pairs in the same structural class.

\begin{corollary}[Dimension comparison after deletion]
\label{cor:dimension-comparison-after-deletion}
Let
$(\mathcal T_1,\mathbf r_1),(\mathcal T_2,\mathbf r_2)
\in[\mathcal T,\mathbf r]_{\mathrm I}^{+}$.  Choose a direction-preserving structural isomorphism that preserves the
smoothness orders and choose
corresponding total orders of the two MIS sets.  Suppose that
Theorem~\ref{thm:homological-deletion} is applied successively along each order,
leaving final remaining sets $\mathcal A_1$ and $\mathcal A_2$.  Then
\[
\begin{aligned}
&\dim \mathcal S_{m,m'}^{\mathbf r_1}(\mathcal T_1)
-
\dim \mathcal S_{m,m'}^{\mathbf r_2}(\mathcal T_2)
\\
={}&
\dim \mathcal H_{m,m'}^{\mathbf r_1}(\mathcal A_1;\mathcal T_1)
-
\dim \mathcal H_{m,m'}^{\mathbf r_2}(\mathcal A_2;\mathcal T_2).
\end{aligned}
\]
In particular, the reference pair $(\mathcal T,\mathbf r)$ is dimensionally stable if and only if the dimension of the remaining correction space is constant on $[\mathcal T,\mathbf r]_{\mathrm I}^{+}$.  Thus such a deletion sequence preserves the spline dimension for each fixed realization and does not change the question of dimensional stability within the structural class.
\end{corollary}

\begin{proof}
By Lemma~\ref{lem:euler-characteristic-structural-invariance}, the two Euler
characteristics are equal.  The homological dimension formula therefore reduces
the difference of spline dimensions to the difference of the two correction
homology dimensions.  Applying Theorem~\ref{thm:homological-deletion}
successively along the chosen orders gives
\[
\dim H_0\!\left(
\mathcal I_{m,m'}^{\mathbf r_k}(\mathcal T_k^\circ)
\right)
=
\dim \mathcal H_{m,m'}^{\mathbf r_k}(\mathcal A_k;\mathcal T_k),
\qquad k=1,2.
\]
Substitution proves the equality and the final equivalence.
\end{proof}

\subsection{Iterative deletion and the weighted CNDC}
\label{subsec:absolute-weight-cndc}

The relative weight in Theorem~\ref{thm:homological-deletion} depends on a total order.  We now give an iterative test that uses only the MISs remaining at the current stage, so no total order is needed in advance.  Reading the deletion history in reverse later gives an order to which Theorem~\ref{thm:homological-deletion} applies.  Throughout the algorithm, the T-mesh is fixed; only the sets of retained and removed MIS summands are updated.

At each stage, let $\mathcal A$ be the set of MISs still under consideration
and let $\mathcal B$ be the set already removed.  We write the stage as
\[
(\mathcal A,\mathcal B)
\]
where
\[
\mathcal A\cap\mathcal B=\varnothing,
\qquad
\mathcal A\cup\mathcal B=\operatorname{MIS}(\mathcal T).
\]
The first entry always denotes the remaining set and the second the removed set.

At a fixed deletion stage, a vertex gives a one-term relation for an MIS exactly when it lies on no other remaining MIS.  We first record these vertices.

\begin{definition}[Current mono-vertices]
\label{def:relative-mono-partition}
Let $\rho\in\mathcal A$.  The current mono-vertices of $\rho$ at the stage
$(\mathcal A,\mathcal B)$ are the vertices on $\rho$ that do not lie on any MIS
in $\mathcal A\setminus\{\rho\}$.  Their set is denoted by
\[
\operatorname{Mono}_{(\mathcal A,\mathcal B)}(\rho).
\]
Equivalently, these vertices are mono-vertices when only the MISs in
$\mathcal A$ are considered.  Thus the set consists of the
original mono-vertices of $\rho$ together with the intersection vertices that
become mono-vertices after the MISs in $\mathcal B$ have been removed.
\end{definition}

Suppose that $\rho$ is removed when the current stage is
$(\mathcal A,\mathcal B)$.  Reverse the deletion history to form a total order of
all MISs, as described below.  Then every vertex in
$\operatorname{Mono}_{(\mathcal A,\mathcal B)}(\rho)$ belongs to
$\Gamma_\iota(\rho)$.  If several MISs are removed in the same pass, they may be
ordered arbitrarily within that pass.  This choice may add intersections to
$\Gamma_\iota(\rho)$, but it cannot remove any current mono-vertex.

The following weight adds the degree contributions of the current mono-vertices.  It gives a sufficient deletion test without requiring a total order in advance.

\begin{definition}[Absolute weight]
\label{def:absolute-weight}
Let $\rho\in\mathcal A$.  Its absolute weight at the state
$(\mathcal A,\mathcal B)$ is
\[
\Omega_{(\mathcal A,\mathcal B)}(\rho)
=
\begin{cases}
\displaystyle
\sum_{\gamma\in\operatorname{Mono}_{(\mathcal A,\mathcal B)}(\rho)}
\mu_h(\gamma),
&
\rho\in\operatorname{MIS}_h(\mathcal T),
\\[4mm]
\displaystyle
\sum_{\gamma\in\operatorname{Mono}_{(\mathcal A,\mathcal B)}(\rho)}
\mu_v(\gamma),
&
\rho\in\operatorname{MIS}_v(\mathcal T).
\end{cases}
\]
The threshold is
\[
\theta(\rho)
=
\begin{cases}
m+1,&\rho\in\operatorname{MIS}_h(\mathcal T),\\
m'+1,&\rho\in\operatorname{MIS}_v(\mathcal T).
\end{cases}
\]
\end{definition}

Suppose the deletion algorithm removes MISs in successive passes.  Reverse the
order of the passes, place the final weighted CNDC last, and choose any order
within each pass.  This gives a total order of all MISs.  For every MIS $\rho$
removed at the stage $(\mathcal A,\mathcal B)$, one has
\[
\operatorname{Mono}_{(\mathcal A,\mathcal B)}(\rho)
\subseteq\Gamma_\iota(\rho),
\qquad
\Omega_{(\mathcal A,\mathcal B)}(\rho)
\le\omega_\iota(\rho).
\]
Thus the absolute weight is a sufficient test: if it reaches the threshold
during the deletion algorithm, then the relative weight reaches the same threshold
in the reversed total order used in Theorem~\ref{thm:homological-deletion}.

Let us illustrate this definition using Fig.~\ref{fig:basic-tmesh}.  Denote
\[
    \rho_h=[v_1,v_2],
    \qquad
    \rho_v=[v_3,v_2]\cup[v_2,v_4].
\]
For the initial stage
\[
    \bigl(\{\rho_h,\rho_v\},\varnothing\bigr),
\]
the vertex $v_2$ is still shared by the two MISs and is not counted as a current
mono-vertex of either MIS.  Hence
\[
    \operatorname{Mono}_{(\{\rho_h,\rho_v\},\varnothing)}(\rho_h)
    =
    \{v_1\},
\]
and
\[
    \operatorname{Mono}_{(\{\rho_h,\rho_v\},\varnothing)}(\rho_v)
    =
    \{v_3,v_4\}.
\]
Therefore
\[
    \Omega_{(\{\rho_h,\rho_v\},\varnothing)}(\rho_h)
    =
    \mu_h(v_1),
\]
and
\[
    \Omega_{(\{\rho_h,\rho_v\},\varnothing)}(\rho_v)
    =
    \mu_v(v_3)+\mu_v(v_4).
\]
After $\rho_v$ has been removed, the current deletion state is
\[
    \bigl(\{\rho_h\},\{\rho_v\}\bigr).
\]
Now $v_2$ becomes a current mono-vertex of $\rho_h$, and hence
\[
    \operatorname{Mono}_{(\{\rho_h\},\{\rho_v\})}(\rho_h)
    =
    \{v_1,v_2\},
\]
so that
\[
    \Omega_{(\{\rho_h\},\{\rho_v\})}(\rho_h)
    =
    \mu_h(v_1)+\mu_h(v_2).
\]
This calculation shows how the absolute weight changes when a shared vertex
becomes a current mono-vertex.

The next proposition shows that corresponding stages on structurally isomorphic
T-meshes have the same absolute weights, provided that the ordered bi-degree is
fixed and corresponding edges have the same smoothness orders.

\begin{proposition}[Invariance of absolute weights]
\label{prop:absolute-weight-structural-invariance}
Let
\[
\alpha:
\widehat{\operatorname{MS}}(\mathcal T)
\longrightarrow
\widehat{\operatorname{MS}}(\mathcal T')
\]
be a direction-preserving structural isomorphism such that corresponding
edges have the same smoothness orders under $\mathbf r$ and $\mathbf r'$.  Then, for any deletion state
$(\mathcal A,\mathcal B)$ of $\operatorname{MIS}(\mathcal T)$ and any
$\rho\in\mathcal A$,
\[
\alpha\!\left(
\operatorname{Mono}_{(\mathcal A,\mathcal B)}(\rho)
\right)
=
\operatorname{Mono}_{(\alpha(\mathcal A),\alpha(\mathcal B))}
\bigl(\alpha(\rho)\bigr),
\]
and
\[
\Omega_{(\mathcal A,\mathcal B)}(\rho)
=
\Omega_{(\alpha(\mathcal A),\alpha(\mathcal B))}
\bigl(\alpha(\rho)\bigr),
\]
where the left- and right-hand weights are computed with $\mathbf r$ and
$\mathbf r'$, respectively.
\end{proposition}

\begin{proof}
Proposition~\ref{prop:induced-tmesh-isomorphism} gives a bijection of interior
vertices and Condition~\textup{(i)} of
Definition~\ref{def:structural-isomorphism} gives
$\alpha(\operatorname{MIS}(\mathcal T))
=\operatorname{MIS}(\mathcal T')$.  Hence a vertex $\gamma$ on $\rho$ lies on
an MIS in $\mathcal A\setminus\{\rho\}$ if and only if $\alpha(\gamma)$ lies on
an MIS in
$\alpha(\mathcal A)\setminus\{\alpha(\rho)\}$.  This proves the equality of the
current mono-vertex sets.

Corresponding maximal segments have the same smoothness orders.  Because the map is direction-preserving, the
values of $\mu_h$ and $\mu_v$ agree separately at corresponding vertices.
Summing the same nonnegative vertex multiplicities over the corresponding sets
proves the equality of the absolute weights.
\end{proof}

We now name the MISs that pass the deletion test.  In this terminology,
``diagonalizable'' means that the available relations generate the coefficient
space of the MIS, so its summand can be removed from the quotient.

\begin{definition}[Diagonalizable MIS]
\label{def:diagonalizable-mis}
Let $(\mathcal A,\mathcal B)$ be a deletion state and let
$\rho\in\mathcal A$.  We say that $\rho$ is diagonalizable at
$(\mathcal A,\mathcal B)$ if
\[
\Omega_{(\mathcal A,\mathcal B)}(\rho)\ge\theta(\rho).
\]
\end{definition}

If an MIS is diagonalizable at some stage, its absolute weight reaches the
threshold.  The total order obtained by reversing the deletion history then gives a relative
weight satisfying
the hypothesis of Theorem~\ref{thm:homological-deletion}.  Hence all MISs removed
by the algorithm can be deleted, in reverse order, without changing the correction
space.  The MISs that remain form the weighted CNDC used below in the dimension
and stability analysis.

In Algorithm~\ref{alg:weighted-cndc}, $\mathcal A$ is the current set of retained MIS summands and $\mathcal B$ is the set already removed from the quotient representation.  The underlying T-mesh is not altered, so the Euler characteristic term in the spline dimension formula is unchanged.

\begin{algorithm}[htb]
\caption{Compute the weighted CNDC.}
\label{alg:weighted-cndc}
\begin{algorithmic}[1]
\REQUIRE The set $\operatorname{MIS}(\mathcal T)$, the bi-degree
$(m,m')$, and a smoothness distribution $\mathbf r$ that is constant along each maximal segment.
\ENSURE The final remaining set $\mathcal A$ and the removed set $\mathcal B$.
\STATE Set $\mathcal A=\operatorname{MIS}(\mathcal T)$ and
$\mathcal B=\varnothing$.
\REPEAT
    \STATE Set $\mathcal Z=\varnothing$.
    \FOR{$\rho\in\mathcal A$}
        \STATE Compute
        $\operatorname{Mono}_{(\mathcal A,\mathcal B)}(\rho)$ and
        $\Omega_{(\mathcal A,\mathcal B)}(\rho)$.
        \IF{$\Omega_{(\mathcal A,\mathcal B)}(\rho)\ge\theta(\rho)$}
            \STATE Set $\mathcal Z=\mathcal Z\cup\{\rho\}$.
        \ENDIF
    \ENDFOR
    \STATE Set $\mathcal A=\mathcal A\setminus\mathcal Z$.
    \STATE Set $\mathcal B=\mathcal B\cup\mathcal Z$.
\UNTIL{$\mathcal Z=\varnothing$}
\STATE Output $(\mathcal A,\mathcal B)$.
\end{algorithmic}
\end{algorithm}

At termination, no MIS in the remaining set satisfies the deletion test.  We now name the two final sets.

\begin{definition}[Weighted CNDC and diagonalizable part]
\label{def:weighted-cndc}
The final remaining set $\mathcal A$ produced by
Algorithm~\ref{alg:weighted-cndc} is called the weighted completely
non-diagonalizable set, or weighted CNDC, and is denoted by
\[
\operatorname{CNDC}_{m,m'}^{\mathbf r}(\mathcal T).
\]
The set $\mathcal B$ of removed MISs is called the diagonalizable part and is
denoted by
\[
\operatorname{Diag}_{m,m'}^{\mathbf r}(\mathcal T).
\]
The weighted CNDC need not be connected; it may be a union of several
T-connected components.  It is the set left by the stated deletion criterion:
no remaining MIS satisfies that criterion.  We do not claim that it is minimal
among all possible quotient representations of the correction homology.
\end{definition}

\begin{remark}[Highest-order smoothness]
\label{rem:highest-order-reduction}
If the smoothness is of highest order in the normal direction, namely
$r(\rho)=m'-1$ on horizontal maximal segments and $r(\rho)=m-1$ on vertical
maximal segments, then $\mu_h(\gamma)=\mu_v(\gamma)=1$ at every relevant
vertex.  The absolute weight is therefore the number of current mono-vertices,
and Algorithm~\ref{alg:weighted-cndc} reduces to the deletion rule based on
relative mono-vertices used in \cite{HuangChen2024}.
\end{remark}

We give a concrete calculation on the T-mesh in Fig.~\ref{fig:basic-tmesh}.  Let
\[
    (m,m')=(3,3),
\]
and suppose that the horizontal maximal segments through
$v_3$ and $v_4$ have smoothness order $1$, while the vertical maximal segments
through $v_1$ and $v_2$ have smoothness order $2$.  Equivalently,
\[
    r_v(v_3)=r_v(v_4)=1,
    \qquad
    r_h(v_1)=r_h(v_2)=2.
\]
Then
\[
    \mu_v(v_3)=\mu_v(v_4)=3-1=2,
    \qquad
    \mu_h(v_1)=\mu_h(v_2)=3-2=1.
\]
At the initial stage
\[
    \mathcal A=\{\rho_h,\rho_v\},
    \qquad
    \mathcal B=\varnothing,
\]
we have
\[
    \Omega_{(\mathcal A,\mathcal B)}(\rho_v)
    =
    \mu_v(v_3)+\mu_v(v_4)
    =
    4.
\]
Since $\rho_v$ is vertical, its threshold is
\[
    \theta(\rho_v)=m'+1=4.
\]
Thus $\rho_v$ is diagonalizable and is removed.  After this
deletion, the current stage is
\[
    \mathcal A=\{\rho_h\},
    \qquad
    \mathcal B=\{\rho_v\}.
\]
Now the current mono-vertices of $\rho_h$ are $v_1$ and $v_2$, and hence
\[
    \Omega_{(\mathcal A,\mathcal B)}(\rho_h)
    =
    \mu_h(v_1)+\mu_h(v_2)
    =
    2.
\]
Since $\rho_h$ is horizontal, its threshold is
\[
    \theta(\rho_h)=m+1=4.
\]
Therefore $\rho_h$ is not diagonalizable.  The algorithm stops with
\[
    \operatorname{CNDC}_{3,3}^{\mathbf r}(\mathcal T)=\{\rho_h\},
    \qquad
    \operatorname{Diag}_{3,3}^{\mathbf r}(\mathcal T)=\{\rho_v\}.
\]
This example shows how a non-uniform smoothness
distribution can make one MIS diagonalizable while leaving another MIS in
the weighted CNDC.

\begin{theorem}[Stopping condition for the weighted CNDC]
\label{thm:cndc-characterization}
Let
\[
\left(
\operatorname{CNDC}_{m,m'}^{\mathbf r}(\mathcal T),
\operatorname{Diag}_{m,m'}^{\mathbf r}(\mathcal T)
\right)
\]
be the final deletion state produced by
Algorithm~\ref{alg:weighted-cndc}.  Then
\[
\Omega_{\left(
\operatorname{CNDC}_{m,m'}^{\mathbf r}(\mathcal T),
\operatorname{Diag}_{m,m'}^{\mathbf r}(\mathcal T)
\right)}(\rho)
<
\theta(\rho)
\]
for every
$\rho\in\operatorname{CNDC}_{m,m'}^{\mathbf r}(\mathcal T)$.
Conversely, the algorithm stops precisely when this strict inequality holds for
every remaining MIS.
\end{theorem}

\begin{proof}
At each pass of Algorithm~\ref{alg:weighted-cndc}, every MIS satisfying
\[
    \Omega_{(\mathcal A,\mathcal B)}(\rho)\ge\theta(\rho)
\]
is placed in $\mathcal Z$.  Whenever $\mathcal Z\neq\varnothing$, at least
one MIS is removed from the finite set $\mathcal A$.  Hence only finitely many
nontrivial passes are possible, and the algorithm terminates.

At the final deletion stage, the stopping condition is
$\mathcal Z=\varnothing$.  Thus no remaining $\rho\in\mathcal A$ satisfies
the deletion inequality, and therefore
\[
    \Omega_{(\mathcal A,\mathcal B)}(\rho)<\theta(\rho)
\]
for every remaining MIS.  Substituting
\[
    \mathcal A=\operatorname{CNDC}_{m,m'}^{\mathbf r}(\mathcal T),
    \qquad
    \mathcal B=\operatorname{Diag}_{m,m'}^{\mathbf r}(\mathcal T)
\]
gives the asserted inequality.  Conversely, if the strict inequality holds
for every $\rho\in\mathcal A$ at some stage, then no remaining MIS is
diagonalizable, so the next pass has
$\mathcal Z=\varnothing$ and the algorithm stops.
\end{proof}

\subsection{Uniqueness, relation rank, stability, and bounds}
\label{subsec:cndc-properties-bounds}

The weighted CNDC is obtained by repeatedly removing MISs whose absolute weights
reach the threshold.  We now prove that the final set is independent of the order
in which eligible MISs are removed and is preserved by direction-preserving
structural isomorphisms.  Thus the weighted CNDC is well defined for a fixed
ordered bi-degree and fixed smoothness orders.

Once some MISs have been removed, additional vertices may become current
mono-vertices of the remaining MISs.  Hence their absolute weights can only
increase.  This monotonicity is the reason why the final set is unique.

\begin{lemma}[Monotonicity of absolute weights]
\label{lem:absolute-weight-monotonicity}
Let
\[
    (\mathcal A_1,\mathcal B_1),
    \qquad
    (\mathcal A_2,\mathcal B_2)
\]
be two states of the deletion process such that
\[
    \mathcal A_2\subseteq \mathcal A_1.
\]
Then, for every $\rho\in\mathcal A_2$,
\[
    \operatorname{Mono}_{(\mathcal A_1,\mathcal B_1)}(\rho)
    \subseteq
    \operatorname{Mono}_{(\mathcal A_2,\mathcal B_2)}(\rho).
\]
Consequently,
\[
    \Omega_{(\mathcal A_1,\mathcal B_1)}(\rho)
    \le
    \Omega_{(\mathcal A_2,\mathcal B_2)}(\rho).
\]
\end{lemma}

\begin{proof}
Let
\[
    \gamma\in
    \operatorname{Mono}_{(\mathcal A_1,\mathcal B_1)}(\rho).
\]
Then $\gamma$ lies on $\rho$ and on no MIS in
$\mathcal A_1\setminus\{\rho\}$.  Since
\[
    \mathcal A_2\setminus\{\rho\}
    \subseteq
    \mathcal A_1\setminus\{\rho\},
\]
the vertex $\gamma$ also lies on no MIS in
$\mathcal A_2\setminus\{\rho\}$.  This proves the inclusion of the two
current mono-vertex sets.

If $\rho$ is horizontal, the corresponding absolute weights are sums of the
nonnegative values $\mu_h(\gamma)$ over these two sets; if $\rho$ is vertical,
the same argument applies with $\mu_v(\gamma)$.  The inclusion therefore gives
\[
    \Omega_{(\mathcal A_1,\mathcal B_1)}(\rho)
    \le
    \Omega_{(\mathcal A_2,\mathcal B_2)}(\rho),
\]
as claimed.
\end{proof}

We now prove uniqueness.  Once an MIS becomes removable, later deletions cannot
make it non-removable.  Hence every removable MIS eventually enters the
diagonalizable part, regardless of when it is selected.

\begin{theorem}[Uniqueness of the weighted CNDC]
\label{thm:cndc-uniqueness}
The weighted CNDC
\[
    \operatorname{CNDC}_{m,m'}^{\mathbf r}(\mathcal T)
\]
is independent of the order in which eligible MISs are removed.
\end{theorem}

\begin{proof}
Run Algorithm~\ref{alg:weighted-cndc} in simultaneous passes and let
\[
\mathcal Z_1,\ldots,\mathcal Z_N
\]
be the nonempty sets removed in the successive passes.  Let
\[
\mathcal C
=
\operatorname{CNDC}_{m,m'}^{\mathbf r}(\mathcal T)
\]
be the final remaining set.  We compare this process with any other process that removes eligible MISs,
one at a time or in groups, until no further deletion is possible.

First we show, by induction on $k$, that every MIS in
$\mathcal Z_1\cup\cdots\cup\mathcal Z_k$ must eventually be removed by the
arbitrary process.  Every $\rho\in\mathcal Z_1$ is removable in the initial
state.  By Lemma~\ref{lem:absolute-weight-monotonicity}, deleting other MISs can
only increase its absolute weight.  Hence $\rho$ remains removable until it is
removed, and the process cannot stop while $\rho$ remains.

Assume the assertion is true through pass $k-1$.  After the other process
has removed all MISs in
$\mathcal Z_1\cup\cdots\cup\mathcal Z_{k-1}$, its remaining set is contained
in the remaining set on which Algorithm~\ref{alg:weighted-cndc} tests the elements of
$\mathcal Z_k$.  Every $\rho\in\mathcal Z_k$ was removable in that larger
remaining set.  Monotonicity therefore shows that it is still removable in the
possibly smaller remaining set of the other process.  Again it remains removable
until it is removed.  Thus every element of $\mathcal Z_k$ must eventually be removed.
This proves that every MIS outside $\mathcal C$ is removed by every process
that continues until no further deletion is possible.

It remains to show that no MIS in $\mathcal C$ can be removed.  We prove
inductively that $\mathcal C$ is contained in the remaining set
$\mathcal A$ at every stage of any such process.  This is true initially.
Assume $\mathcal C\subseteq\mathcal A$.  For
$\rho\in\mathcal C$, Lemma~\ref{lem:absolute-weight-monotonicity}, applied with
the final remaining set $\mathcal C$ as the smaller set, gives
\[
\Omega_{(\mathcal A,\mathcal B)}(\rho)
\le
\Omega_{\left(
\mathcal C,
\operatorname{MIS}(\mathcal T)\setminus\mathcal C
\right)}(\rho)
<
\theta(\rho),
\]
where the strict inequality is the stopping condition of
Algorithm~\ref{alg:weighted-cndc}.  Hence no element of $\mathcal C$ is
removable at that state, so the inclusion remains true after the next deletion.

Every process that continues until no further deletion is possible therefore
removes exactly the MISs outside
$\mathcal C$ and keeps exactly the MISs in $\mathcal C$.  The set of remaining MISs is unique.
\end{proof}

The next theorem shows that the construction is preserved by a direction-preserving
structural isomorphism when corresponding edges have the same smoothness orders.

\begin{theorem}[Structural invariance of the weighted CNDC]
\label{thm:cndc-structural-invariance}
Let
\[
\alpha:
\widehat{\operatorname{MS}}(\mathcal T)
\longrightarrow
\widehat{\operatorname{MS}}(\mathcal T')
\]
be a direction-preserving structural isomorphism such that corresponding
edges have the same smoothness orders under $\mathbf r$ and $\mathbf r'$.  Then
\[
\alpha\!\left(
\operatorname{CNDC}_{m,m'}^{\mathbf r}(\mathcal T)
\right)
=
\operatorname{CNDC}_{m,m'}^{\mathbf r'}(\mathcal T').
\]
Thus corresponding pairs in the same structural class have corresponding
weighted CNDCs.
\end{theorem}

\begin{proof}
Let $(\mathcal A,\mathcal B)$ be any state of
Algorithm~\ref{alg:weighted-cndc} for $\mathcal T$.  Its image
$(\alpha(\mathcal A),\alpha(\mathcal B))$ is the corresponding state for
$\mathcal T'$, because the structural isomorphism preserves MIS type.  By
Proposition~\ref{prop:absolute-weight-structural-invariance}, for every
$\rho\in\mathcal A$,
\[
\Omega_{(\mathcal A,\mathcal B)}(\rho)
=
\Omega_{(\alpha(\mathcal A),\alpha(\mathcal B))}
\bigl(\alpha(\rho)\bigr).
\]
The fixed ordered bi-degree and the equality of corresponding smoothness
orders preserve the thresholds.  Hence $\rho$ satisfies the deletion condition if and only if
$\alpha(\rho)$ does.

Starting from
\[
(\operatorname{MIS}(\mathcal T),\varnothing),
\qquad
(\operatorname{MIS}(\mathcal T'),\varnothing),
\]
an induction over the passes of the algorithm shows that the sets removed at
each pass correspond under $\alpha$.  The final remaining sets therefore
correspond, proving the result.
\end{proof}

We now return to the spline dimension.  Initially, all MISs appear in the quotient representation of the correction homology.  By reversing the deletion history and applying Theorem~\ref{thm:homological-deletion}, we can remove every MIS in the diagonalizable part.  The same correction space is then represented using only the weighted CNDC.  The Euler characteristic term remains that of the quotient complex on the original T-mesh, and all relations from the original T-mesh are retained.

\begin{theorem}[Dimension formula using the weighted CNDC]
\label{thm:cndc-dimension-formula}
One has
\[
\begin{aligned}
\dim \mathcal S_{m,m'}^{\mathbf r}(\mathcal T)
={}&
\chi\!\left(
\mathcal Q_{m,m'}^{\mathbf r}(\mathcal T^\circ)
\right)
\\
&+
\dim
\mathcal H_{m,m'}^{\mathbf r}
\left(
\operatorname{CNDC}_{m,m'}^{\mathbf r}(\mathcal T);
\mathcal T
\right).
\end{aligned}
\]
The Euler characteristic term is evaluated on the original T-mesh $\mathcal T$.  No new quotient complex is associated with the weighted CNDC.
\end{theorem}

\begin{proof}
By Theorem~\ref{thm:homological-dimension-formula},
\[
\begin{aligned}
\dim \mathcal S_{m,m'}^{\mathbf r}(\mathcal T)
={}&
\chi\!\left(
\mathcal Q_{m,m'}^{\mathbf r}(\mathcal T^\circ)
\right)
+
\dim H_0\!\left(
\mathcal I_{m,m'}^{\mathbf r}(\mathcal T^\circ)
\right).
\end{aligned}
\]
Proposition~\ref{prop:mis-representation} identifies the correction
homology with
$\mathcal H_{m,m'}^{\mathbf r}(\operatorname{MIS}(\mathcal T);\mathcal T)$.

Let $\mathcal Z_1,\ldots,\mathcal Z_N$ be the nonempty sets removed in the
successive passes of Algorithm~\ref{alg:weighted-cndc}.  Form a total
order by listing first the MISs in $\mathcal Z_N$, then those in
$\mathcal Z_{N-1}$, and so on down to $\mathcal Z_1$, with arbitrary order
inside each $\mathcal Z_k$, and finally listing the MISs in the weighted CNDC.
Consider $\rho\in\mathcal Z_k$.  At the pass in which $\rho$ is removed, its
absolute weight satisfies
\[
    \Omega_{(\mathcal A,\mathcal B)}(\rho)\ge\theta(\rho).
\]
As observed after Definition~\ref{def:absolute-weight}, its relative vertex set
in this total order
contains the current mono-vertex set used by the algorithm.  Hence
\[
    \omega_\iota(\rho)
    \ge
    \Omega_{(\mathcal A,\mathcal B)}(\rho)
    \ge
    \theta(\rho).
\]
Theorem~\ref{thm:homological-deletion} can therefore be applied successively to
all MISs in
$\operatorname{Diag}_{m,m'}^{\mathbf r}(\mathcal T)$ in this reverse order of deletion.
Because the correction space retains all induced relations, every step gives an
isomorphism, and the final remaining set is exactly
$\operatorname{CNDC}_{m,m'}^{\mathbf r}(\mathcal T)$.  Consequently,
\[
    \dim H_0\!\left(
    \mathcal I_{m,m'}^{\mathbf r}(\mathcal T^\circ)
    \right)
    =
    \dim \mathcal H_{m,m'}^{\mathbf r}
    \left(
    \operatorname{CNDC}_{m,m'}^{\mathbf r}(\mathcal T);\mathcal T
    \right).
\]

Substitution into the homological dimension formula proves the result.
\end{proof}

\begin{corollary}[Dimensional stability and the remaining relations]
\label{cor:cndc-stability-characterization}
For each
$(\mathcal T',\mathbf r')\in[\mathcal T,\mathbf r]_{\mathrm I}^{+}$, set
\[
\mathcal C(\mathcal T',\mathbf r')
:=
\operatorname{CNDC}_{m,m'}^{\mathbf r'}(\mathcal T').
\]
Then $(\mathcal T,\mathbf r)$ is dimensionally stable if and only if
\[
\dim
\mathcal H_{m,m'}^{\mathbf r'}
\bigl(\mathcal C(\mathcal T',\mathbf r');\mathcal T'\bigr)
\]
is constant on $[\mathcal T,\mathbf r]_{\mathrm I}^{+}$.

Equivalently, after choosing corresponding bases for the spaces
$\mathcal M(\mathcal C(\mathcal T',\mathbf r'))$, dimensional stability
holds if and only if
\[
\operatorname{rank}
K_{\mathcal C(\mathcal T',\mathbf r')}(\mathcal T')
\]
is constant on the structural class.
\end{corollary}

\begin{proof}
By Lemma~\ref{lem:euler-characteristic-structural-invariance}, the
Euler characteristic term in Theorem~\ref{thm:cndc-dimension-formula} is
constant on the structural class.  By
Theorem~\ref{thm:cndc-structural-invariance}, corresponding weighted CNDCs have
the same MIS types and the same smoothness orders.  Hence
$\dim\mathcal M(\mathcal C(\mathcal T',\mathbf r'))$ is also constant.
Theorem~\ref{thm:cndc-dimension-formula} and
\[
\dim\mathcal H
=
\dim\mathcal M-\operatorname{rank}K
\]
therefore give the two equivalences.
\end{proof}

\begin{remark}[Role of the remaining relation space]
\label{rem:cndc-determines-stability}
The weighted CNDC identifies the MIS summands needed to represent the remaining
correction term.  It does not, by itself, determine the relation space among those
summands.  The latter is
\[
\mathcal R_{\mathcal C}(\mathcal T)
=
\mathcal M(\mathcal C)\cap\mathcal R(\mathcal T),
\]
Because $\mathcal R_{\mathcal C}(\mathcal T)$ is obtained by intersecting with
the full relation space $\mathcal R(\mathcal T)$, it also contains relations
produced by eliminating removed summands.  Their coefficients may depend on mesh
lines outside the weighted CNDC.  The results above reduce the stability problem
to the dimension of this relation space; they do not show that it is determined
only by the coordinates of the MISs in the weighted CNDC.
\end{remark}

The following example shows that the weighted CNDC alone does not determine the correction term: the same CNDC can lead to different ranks for the remaining relations.  The example has highest-order smoothness and is therefore a special case of the present setting.

\begin{example}[A nonempty CNDC with geometry-dependent rank]
\label{ex:unstable-cndc}
Let $m=m'=3$ and impose smoothness order $2$ on every interior edge.  Consider
the T-mesh in Fig.~\ref{fig:unstable-cndc}.  Its four thick interior segments are
the four MISs.  In the highest-order terminology of \citet{HuangChen2024}, each
of these MISs has fewer than $d+1=4$ mono-vertices, so no MIS is removed by the
highest-order deletion rule.  Hence the CNDC is the T-connected component formed by all four MISs.

\begin{figure}[htb]
\centering
\begin{tikzpicture}[scale=0.82]
    \draw (1,1) rectangle (6,6);
    \draw (1,5) -- (6,5);
    \draw (1,2) -- (6,2);
    \draw (2,1) -- (2,6);
    \draw (5,1) -- (5,6);
    \draw (1.5,1) -- (1.5,6);
    \draw (1,5.5) -- (6,5.5);
    \draw (1,1.5) -- (6,1.5);
    \draw (5.5,1) -- (5.5,6);

    \draw[line width=1.5pt] (3,5.5) -- (3,2);
    \draw[line width=1.5pt] (1.5,4) -- (5,4);
    \draw[line width=1.5pt] (2,3) -- (5.5,3);
    \draw[line width=1.5pt] (4,1.5) -- (4,5);

    \fill (3,5.5) circle (1.8pt) node[above right] {$v_1$};
    \fill (4,5) circle (1.8pt) node[above right] {$v_2$};
    \fill (1.5,4) circle (1.8pt) node[above right] {$v_3$};
    \fill (5,4) circle (1.8pt) node[above right] {$v_4$};
    \fill (2,3) circle (1.8pt) node[above right] {$v_5$};
    \fill (5.5,3) circle (1.8pt) node[above right] {$v_6$};
    \fill (3,2) circle (1.8pt) node[above right] {$v_7$};
    \fill (4,1.5) circle (1.8pt) node[below right] {$v_8$};
\end{tikzpicture}
\caption{A highest-order smoothness T-mesh with a nonempty CNDC.  The four thick segments are its MISs.}
\label{fig:unstable-cndc}
\end{figure}
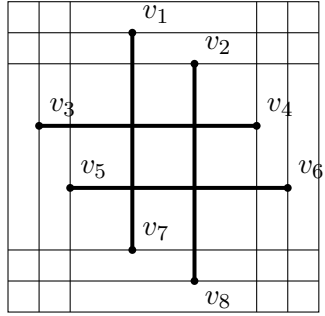

Let the relevant vertical line coordinates be
$s_2<\cdots<s_7$ and the relevant horizontal line coordinates be
$t_2<\cdots<t_7$.  The smoothing-cofactor calculation for this example
\citep{LiChen2011,HuangChen2024} reduces the full conformality matrix $M_{\mathrm{conf}}$ by
invertible row and column operations to an identity block of size $12$ and the
matrix
\[
\widetilde M=
\begin{pmatrix}
1 & a & 0 & 0\\
b & 0 & 1 & 0\\
0 & 0 & c & 1\\
0 & d & 0 & 1
\end{pmatrix},
\]
where
\[
\begin{aligned}
a&=f_{3,6,7}(t_{4,5}),&
b&=f_{2,3,6}(s_{4,5}),\\
c&=f_{2,3,6}(t_{5,4}),&
d&=f_{3,6,7}(s_{4,5}),
\end{aligned}
\]
and
\[
f_{i,j,k}(x_{p,q})
=
\frac{(x_p-x_i)(x_p-x_j)(x_p-x_k)}
     {(x_q-x_i)(x_q-x_j)(x_q-x_k)}.
\]
Thus
\[
\det\widetilde M=d-cba,
\qquad
\operatorname{rank}M_{\mathrm{conf}}
=
12+\operatorname{rank}\widetilde M.
\]
The corresponding spline dimension is
\[
\dim S_3(\mathcal T)
=
\begin{cases}
65,&\det\widetilde M=0,\\
64,&\det\widetilde M\ne0.
\end{cases}
\]

Both cases occur while the mesh structure is unchanged.  For example, take
\[
(t_2,t_3,t_4,t_5,t_6,t_7)=(0,1,2,3,4,5)
\]
and initially
\[
(s_2,s_3,s_4,s_5,s_6,s_7)=(0,1,2,3,4,5).
\]
Then
\[
a=\frac32,\qquad
b=\frac23,\qquad
c=\frac32,\qquad
d=\frac32,
\]
so $\det\widetilde M=0$ and the dimension is $65$.  If only $s_7$ is
changed from $5$ to $6$, the order of the vertical lines is unchanged,
$b=2/3$, while $d=4/3$.  Hence
\[
\det\widetilde M
=
\frac43-\frac32\cdot\frac23\cdot\frac32
=
-\frac16\ne0,
\]
and the dimension is $64$.

This example has the same nonempty CNDC in both realizations, but the matrix of
the remaining relations has different ranks.  Thus the weighted CNDC is fixed by
the structure, whereas the rank may still depend on the geometric realization.
\end{example}

We next bound the remaining correction term.  Let
$\mathcal Z_1,\ldots,\mathcal Z_N$ be the nonempty sets removed in the
successive passes of Algorithm~\ref{alg:weighted-cndc}, and put
\[
\mathcal C
:=
\operatorname{CNDC}_{m,m'}^{\mathbf r}(\mathcal T).
\]
To apply the ordered estimate, reverse the deletion history.  We say that a total
order $\widehat\iota$ of $\operatorname{MIS}(\mathcal T)$ follows this reversed
history if it lists first the MISs in $\mathcal Z_N$, then those in
$\mathcal Z_{N-1}$, and so on down to $\mathcal Z_1$, and finally the MISs in
$\mathcal C$.  The order within each $\mathcal Z_k$ and within $\mathcal C$
may be chosen arbitrarily.  Any order of the weighted CNDC can therefore be
extended to an order of this form.  All relative vertex sets and weights below
are computed in the resulting order of the full MIS set.

For each MIS $\rho$, define
\[
\beta^{\mathbf r}(\rho)
:=
\begin{cases}
\bigl(m'-r(\rho)\bigr)_+,
&\rho\in\operatorname{MIS}_h(\mathcal T),\\[1mm]
\bigl(m-r(\rho)\bigr)_+,
&\rho\in\operatorname{MIS}_v(\mathcal T).
\end{cases}
\]
This is the dimension of the transverse factor of $\mathcal B_\rho$.

\begin{theorem}[Upper bound for the remaining correction term]
\label{thm:cndc-upper-bound}
For every total order $\widehat\iota$ that follows the reversed deletion
history, one has
\[
\begin{aligned}
0
\le{}&
\dim\mathcal H_{m,m'}^{\mathbf r}(\mathcal C;\mathcal T)
\\
\le{}&
\sum_{\substack{\rho\in\mathcal C\\
\rho\in\operatorname{MIS}_h(\mathcal T)}}
\bigl(m+1-\omega_{\widehat\iota}(\rho)\bigr)_+
\beta^{\mathbf r}(\rho)
\\
&+
\sum_{\substack{\rho\in\mathcal C\\
\rho\in\operatorname{MIS}_v(\mathcal T)}}
\bigl(m'+1-\omega_{\widehat\iota}(\rho)\bigr)_+
\beta^{\mathbf r}(\rho).
\end{aligned}
\]
Consequently,
\[
\begin{aligned}
\dim\mathcal S_{m,m'}^{\mathbf r}(\mathcal T)
\le{}&
\chi\!\left(
\mathcal Q_{m,m'}^{\mathbf r}(\mathcal T^\circ)
\right)
\\
&+
\sum_{\substack{\rho\in\mathcal C\\
\rho\in\operatorname{MIS}_h(\mathcal T)}}
\bigl(m+1-\omega_{\widehat\iota}(\rho)\bigr)_+
\beta^{\mathbf r}(\rho)
\\
&+
\sum_{\substack{\rho\in\mathcal C\\
\rho\in\operatorname{MIS}_v(\mathcal T)}}
\bigl(m'+1-\omega_{\widehat\iota}(\rho)\bigr)_+
\beta^{\mathbf r}(\rho).
\end{aligned}
\]
\end{theorem}

\begin{proof}
Write such a total order as
\[
\widehat\iota:\quad
\rho_1\succ\rho_2\succ\cdots\succ\rho_t
\]
and set
\[
\mathcal A_i:=\{\rho_i,\ldots,\rho_t\},\qquad
\mathcal V_i:=\mathcal M(\mathcal A_i),\qquad
\mathcal R_i:=\mathcal V_i\cap\mathcal R(\mathcal T).
\]
Let $\mathcal V_{t+1}=\mathcal R_{t+1}=0$, and let
\[
\pi_i:\mathcal V_i
=
[\rho_i]\mathcal B_{\rho_i}\oplus\mathcal V_{i+1}
\longrightarrow
[\rho_i]\mathcal B_{\rho_i}
\]
be the projection.  The inclusion of $\mathcal V_{i+1}$ and the projection
$\pi_i$ induce a short exact sequence
\[
0\longrightarrow
\mathcal V_{i+1}/\mathcal R_{i+1}
\longrightarrow
\mathcal V_i/\mathcal R_i
\longrightarrow
[\rho_i]\mathcal B_{\rho_i}/\pi_i(\mathcal R_i)
\longrightarrow0.
\]
Indeed,
$\mathcal R_{i+1}=\mathcal V_{i+1}\cap\mathcal R_i$, and a class lies in the
kernel of the last map exactly when a representative can be changed by an
element of $\mathcal R_i$ to lie in $\mathcal V_{i+1}$.

Suppose first that $\rho_i$ is horizontal.  Because
$\widehat\iota$ is a total order of all MISs, every vertex in
$\Gamma_{\widehat\iota}(\rho_i)$ is either an original mono-vertex or the
intersection of $\rho_i$ with an MIS belonging to $\mathcal A_{i+1}$.  Hence
the corresponding one-term or two-term vertex relation lies in
$\mathcal R_i$.  Its projection to the $\rho_i$-summand shows that
$\pi_i(\mathcal R_i)$ contains
\[
[\rho_i]\left(
\sum_{\gamma\in\Gamma_{\widehat\iota}(\rho_i)}
(s-s_\gamma)^{r_h(\gamma)+1}
\mathbb R[s]_{\le m-r_h(\gamma)-1}
\right)
\otimes
\mathbb R[t]_{\le m'-r(\rho_i)-1}.
\]
The univariate dimension formula used in
Theorem~\ref{thm:homological-deletion} therefore gives
\[
\dim
[\rho_i]\mathcal B_{\rho_i}/\pi_i(\mathcal R_i)
\le
\bigl(m+1-\omega_{\widehat\iota}(\rho_i)\bigr)_+
\beta^{\mathbf r}(\rho_i).
\]
For a vertical $\rho_i$, interchanging $s$ and $t$ gives
\[
\dim
[\rho_i]\mathcal B_{\rho_i}/\pi_i(\mathcal R_i)
\le
\bigl(m'+1-\omega_{\widehat\iota}(\rho_i)\bigr)_+
\beta^{\mathbf r}(\rho_i).
\]

Taking dimensions in these short exact sequences and summing over
$i=1,\ldots,t$ gives the corresponding upper bound obtained from all MISs for
\[
\dim\mathcal H_{m,m'}^{\mathbf r}
\bigl(\operatorname{MIS}(\mathcal T);\mathcal T\bigr).
\]
For every
$\rho\in\operatorname{Diag}_{m,m'}^{\mathbf r}(\mathcal T)$, the construction of
$\widehat\iota$ and the observation after Definition~\ref{def:absolute-weight}
give
\[
\omega_{\widehat\iota}(\rho)\ge\theta(\rho).
\]
The corresponding term in the upper bound is therefore zero.  By
Theorem~\ref{thm:cndc-dimension-formula}, the correction space defined using all MISs has
the same dimension as the correction space on $\mathcal C$, so only the displayed
summands indexed by the weighted CNDC remain.  The lower bound is immediate from nonnegativity of
dimension, and substitution into
Theorem~\ref{thm:cndc-dimension-formula} proves the upper bound for the spline dimension.
\end{proof}

\begin{corollary}[Comparison with Mourrain's upper bound]
\label{cor:cndc-mourrain-comparison}
Let $U_{\mathrm{all}}(\widehat\iota)$ denote the sum obtained in the proof
above when the estimate is applied to all
$\operatorname{MIS}(\mathcal T)$, and let
$U_{\mathrm{CNDC}}(\widehat\iota)$ denote the two sums in
Theorem~\ref{thm:cndc-upper-bound}.  For every total order that follows the
reversed deletion history,
\[
U_{\mathrm{CNDC}}(\widehat\iota)
=
U_{\mathrm{all}}(\widehat\iota).
\]
\end{corollary}

\begin{proof}
Every MIS in the diagonalizable part has relative weight at least
its threshold in any order that follows the reversed deletion history, so its
term in the sum over all MISs vanishes.  The remaining summands are precisely those indexed by
$\mathcal C$.
\end{proof}

\begin{remark}
For a fixed order that follows the reversed deletion history, the bound in
Theorem~\ref{thm:cndc-upper-bound} has the same numerical value as Mourrain's
corresponding bound over all MISs, because every deleted MIS already contributes
zero.  We therefore do not claim a sharper numerical bound.  The advantage is
that the possible nonzero terms are identified in advance: only MISs in the
weighted CNDC can contribute, although their weights are computed in the chosen
order of the full MIS set.
\end{remark}

\begin{corollary}[Empty weighted CNDC and dimensional stability]
\label{cor:cndc-stability}
If
\[
\operatorname{CNDC}_{m,m'}^{\mathbf r}(\mathcal T)=\varnothing,
\]
then, for every
$(\mathcal T',\mathbf r')\in[\mathcal T,\mathbf r]_{\mathrm I}^{+}$,
\[
\dim\mathcal S_{m,m'}^{\mathbf r'}(\mathcal T')
=
\chi\!\left(
\mathcal Q_{m,m'}^{\mathbf r'}((\mathcal T')^\circ)
\right).
\]
Hence the reference pair $(\mathcal T,\mathbf r)$ is dimensionally stable in
the sense of Definition~\ref{def:dimensional-stability}.  Thus an empty weighted CNDC is a directly checkable sufficient condition for stability.
\end{corollary}

\begin{proof}
Let
$(\mathcal T',\mathbf r')\in[\mathcal T,\mathbf r]_{\mathrm I}^{+}$ and choose
a structural isomorphism satisfying Definition~\ref{def:structural-class}.
Theorem~\ref{thm:cndc-structural-invariance} gives
\[
\operatorname{CNDC}_{m,m'}^{\mathbf r'}(\mathcal T')=\varnothing.
\]
Theorem~\ref{thm:cndc-dimension-formula} then gives the displayed equality.  By
Lemma~\ref{lem:euler-characteristic-structural-invariance}, the Euler
characteristic is the same as that of the reference pair.  Thus the spline
dimension is constant on the structural class.
\end{proof}

\section{Conclusion}
\label{sec:conclusion}

In this paper, we have studied the dimension and dimensional stability of polynomial spline spaces over planar T-meshes whose smoothness order is constant along each maximal segment.  The homological dimension formula separates the dimension into the Euler characteristic of a quotient complex and a correction homology term.

The two terms play different roles.  The Euler characteristic term is determined by the numbers of cells, horizontal and vertical interior edges, and interior vertices, together with the ordered bi-degree and the smoothness orders.  It is constant on the structural class.  It also remains fixed during the deletion process, because this process changes neither the T-mesh nor any chain complex; it only removes summands from the quotient representation of the correction homology.

The main result is a weighted deletion theorem for MIS summands.  Repeated deletion produces the weighted CNDC.  This set is independent of the order in which eligible MISs are removed and is the same for corresponding pairs in the structural class.  Hence
\[
\dim \mathcal S_{m,m'}^{\mathbf r}(\mathcal T)
=
\chi\!\left(\mathcal Q_{m,m'}^{\mathbf r}(\mathcal T^\circ)\right)
+
\dim \mathcal H_{m,m'}^{\mathbf r}
\left(
\operatorname{CNDC}_{m,m'}^{\mathbf r}(\mathcal T);\mathcal T
\right).
\]
The first term is still computed on the original T-mesh, while the second term is represented using only the MISs in the weighted CNDC.

The weighted CNDC reduces the number of MIS summands, but it does not determine the correction term by itself.  The remaining relation space is
\[
\mathcal R_{\mathcal C}(\mathcal T)
=
\mathcal M(\mathcal C)\cap\mathcal R(\mathcal T),
\]
so relations induced by the removed summands are retained.  After bases are fixed, dimensional stability is equivalent to constancy of the dimension of the remaining relation space, or equivalently of the rank of any matrix representing these relations.  The example in Section~\ref{sec:decomposition} shows that the weighted CNDC may remain unchanged while this rank and the spline dimension vary.  Thus a nonempty weighted CNDC does not by itself imply either stability or instability.

An empty weighted CNDC gives a direct sufficient condition for dimensional stability.  The upper bound obtained from an order that follows the reversed deletion history has the same numerical value as Mourrain's corresponding bound over all MISs, because the removed MISs contribute zero.  The numerical value is not sharper, but the formula identifies in advance the MISs that may give nonzero terms.

A remaining problem is to identify classes of nonempty weighted CNDCs for which the dimension of the remaining relation space is fixed by the mesh structure and the prescribed smoothness orders.  Such a result would give further stability criteria and may lead to smaller matrices for dimension computation.

\bibliographystyle{plainnat}
\bibliography{references}

\end{document}